\documentclass[11pt]{article}
\usepackage{amsfonts}
\usepackage{makeidx}
\usepackage{amsthm,amssymb}
\usepackage{latexsym}
\usepackage{graphicx}
\usepackage{amsmath}
\usepackage[T1]{fontenc}
\usepackage{amsfonts}
\usepackage[normalem]{ulem}

\newtheorem{theorem}{Theorem}

\newtheorem{claim}[theorem]{Claim}

\newtheorem{adefinition}[theorem]{Definition}
\newenvironment{definition}{\begin{adefinition}\rm}{\end{adefinition}}
\newtheorem{aexample}[theorem]{Example}

\newtheorem{lemma}[theorem]{Lemma}

\newtheorem{proposition}[theorem]{Proposition}
\newtheorem{aremark}[theorem]{Remark}
\newenvironment{remark}{\begin{aremark}\rm}{\end{aremark}}

\numberwithin{equation}{section} \numberwithin{theorem}{section}
\DeclareMathOperator\Tr{Tr}

\newcommand{\N}{\mathbb{N}}
\newcommand{\NN}{\mathcal{N}}

\newcommand{\G}{\mathcal{G}}
\newcommand{\E}{\mathbf{E}}
\newcommand{\V}{\mathbf{Var}}
\newcommand{\R}{\mathbb{R}}
\newcommand{\C}{\mathbb{C}}
\newcommand{\al}{\alpha}
\newcommand{\M}{\mathcal{M}_n}
\newcommand{\MM}{\mathcal{M}_{n,m,k}}
\newcommand{\MMy}{\mathcal{M}_{n,m,k}(\mathbf{y})}
\newcommand{\HH}{\mathcal{H}}
\newcommand{\Y}{{Y}}
\newcommand{\YY}{\mathcal{Y}_n}
\newcommand{\y}{\mathbf{y}}
\newcommand{\pp}{\mathbf{p}}
\newcommand{\jj}{\mathbf{j}}
\newcommand{\ii}{\mathbf{i}}
\newcommand{\q}{\mathbf{q}}
\newcommand{\s}{\mathbf{s}}

\newcommand{\aaa}{{a}}
\newcommand{\bb}{{b}}
\newcommand{\ttt}{\mathbf{t}}

\makeatother

\begin{document}

\title{CLT for linear
eigenvalue statistics for a tensor product version of sample covariance matrices}
\author{A. Lytova \\
{\small {\textit{Department of Mathematical and Statistical Sciences,
University of Alberta}} }\\
{\small {\textit{Edmonton, Alberta, Canada, T6G 2G1}} }\\
{\small {\textit{Institute of Mathematics and Informatics, Opole University}} }\\
{\small {\textit{Opole, Poland, 45 052}} }\\
{\small {E-mail: alytova@math.uni.opole.pl } }\\
}
\date{}
\maketitle

\begin{abstract}
For $k,m,n\in \N$, we consider $n^k\times n^k$  random matrices of the form
$$
\MMy=\sum_{\alpha=1}^m\tau_\alpha {Y_\alpha}Y_\alpha^T,\quad
\Y_\al=\mathbf{y}_\al^{(1)}\otimes...\otimes\mathbf{y}_\al^{(k)},
$$
where $\tau _{\alpha }$, $\al\in[m]$, are  real numbers and
$\mathbf{y}_\al^{(j)}$, $\al\in[m]$, $j\in[k]$,  are i.i.d. copies of  a normalized isotropic random vector
$\mathbf{y}\in \mathbb{R}^n$.
For every fixed $k\ge 1$, if the Normalized Counting Measures of
$\{\tau _{\alpha }\}_{\alpha}$ converge weakly as $m,n\rightarrow \infty$,  $m/n^k\rightarrow
c\in \lbrack 0,\infty )$ and
$\y$ is a {good} vector in the sense of
Definition \ref{d:good}, then the Normalized Counting Measures of
eigenvalues of $\MMy$ converge weakly in probability to a
non-random limit found in \cite{Ma-Pa:67}. For $k=2$, we  define a subclass of good vectors  $\y$ for which the centered linear eigenvalue
statistics  $n^{-1/2}\Tr \varphi(\mathcal{M}_{n,m,2}(\mathbf{y}))^\circ$ converge in distribution to
a Gaussian random variable, i.e., the Central Limit Theorem is valid.
\end{abstract}

\maketitle

\section{Introduction: Problem and Main Result}
\label{s:intro}

For every $k\in \mathbb{N}$, consider random vectors of the form
\begin{equation}
\label{Y}
\Y=\mathbf{y}^{(1)}\otimes...\otimes\mathbf{y}^{(k)}\in (\mathbb{R}^n)^{\otimes k},
\end{equation}
where $\mathbf{y}^{(1)}$,..., $\mathbf{y}^{(k)}$ are i.i.d. copies of {\it a normalized isotropic} random vector
$\mathbf{y}=(y_1,...,y_n)\in \mathbb{R}^n$,
\begin{equation}
\label{iso}
\mathbf{E}\{y_j\}=0, \quad\mathbf{E}\{y_i y_j\}=\delta_{ij}n^{-1}, \quad i,j\in[n],
\end{equation}
$[n]=\{1,...,n\}$. The components of $\Y$ have the form
\begin{equation*}
 {\Y}_ \mathbf{j}=y_{j_1}^{(1)}\times...\times y_{j_k}^{(k)},
\end{equation*}
where we use the notation $\mathbf{j}$ for $k$-multiindex:
\begin{equation*}
  \mathbf{j}=\{j_1,...,j_k\},\quad j_1,...,j_k\in [n].
\end{equation*}
For every $m\in \mathbb{N}$, let $\{\Y_\alpha\}_{\alpha =1}^{m}$ be i.i.d. copies of $\Y$, and let $\{\tau _{\alpha }\}_{\alpha =1}^{m}$ be a collection
of real numbers. Consider an $n^k\times n^k$ real symmetric random matrix corresponding to a  normalized isotropic random vector
$\mathbf{y}$,
\begin{equation}
\label{M}
{\mathcal{M}_n=\MM=\MM(\y)=\sum_{\alpha=1}^m\tau_\alpha {{\Y}_\alpha}{{\Y}_\alpha}^T}.
\end{equation}
We suppose that
\begin{equation}\label{mn}
m\rightarrow \infty\quad \text{and}\quad m/n^k\rightarrow c\in(0,\infty) \quad \text{as}\quad n\rightarrow\infty.
\end{equation}
Note that  $\MM$ can be also written in the form
\begin{equation}
\mathcal{M}_{n,m,k}={B}_{n,m,k}T_m {B}_{n,m,k}^T, \label{MT}
\end{equation}
where
\begin{equation*}
{B}_{n,m,k}=(\Y_1\;\;\Y_2\;\;...\;\;\Y_m),\quad T_m=\{\tau_\alpha \delta_{\alpha\beta}\}_{\alpha,\beta=1}^m.
\end{equation*}

Such matrices with  $T_m\ge 0$ (not necessarily diagonal) are known as sample covariance matrices. The asymptotic behavior of their spectral statistics is well studied when all entries of $Y_\al$ are independent.
Much less is known in the case when columns $Y_\al$ have dependence in their structure.

The model constructed in (\ref{M})  appeared  in {\it the quantum information theory}, and was introduced to random matrix theory by Hastings (see \cite{AHH:12,Ha:09,Ha:07}). In  \cite{AHH:12}, it was studied as a quantum analog of the classical
 probability problem on the allocation of $p$ balls among $q$ boxes (a quantum
 model of data hiding and correlation locking scheme). In particular,  by combinatorial analysis
of moments of   $n^{-k}\Tr \M^p$, $p\in \mathbb{N}$, it was proved that for the special cases of random vectors $\mathbf{y}$ uniformly distributed on the unit sphere in $\mathbb{C}^n$ or having Gaussian components, the expectations of the Normalized Counting
Measures of eigenvalues of the corresponding matrices converge to the Marchenko-Pastur
law \cite{Ma-Pa:67}.
The main goal of the present paper is to extend this result of \cite{AHH:12}
to a wider class of matrices $M_{n,m,k}(\y)$  and also to prove the Central Limit Theorem  for linear eigenvalue statistics  in the case $k=2$.

Let $\{\lambda^{(n)}_{l}\}_{l=1}^{n^k}$ be the eigenvalues of $\M$
counting
their multiplicity, and introduce their Normalized Counting Measure (NCM) $%
N_{n}$, setting for every $\Delta \subset \mathbb{R}$
\begin{equation*}
N_{n}(\Delta )=\mathrm{Card}\{l\in [n^k]:\lambda^{(n)} _{l}\in
\Delta \}/n^k.  
\end{equation*}
Likewise, define the NCM $\sigma _{m}$ of $\{\tau _{\alpha }\}_{\alpha
=1}^{m}$,
\begin{equation}
\sigma _{m}(\Delta )=\mathrm{Card}\{\alpha\in \lbrack
m]:\tau_\alpha\in\Delta\}/m.  \label{sigm}
\end{equation}
We assume that the sequence $\{\sigma _{m}\}_{m=1}^\infty$ converges
weakly:
\begin{equation}
\lim_{m\rightarrow \infty }\sigma _{m}=\sigma, \; \sigma (\mathbb{R})=1.
\label{sigma}
\end{equation}
In the case $k=1$, there is a number of papers devoted to the convergence of the NCMs of the eigenvalues of $\mathcal{M}_{n,m,1}$ and related matrices  (see \cite{Adam:11}, \cite{Ba-Zh:08},  \cite{G-N-T:01}, \cite{Ma-Pa:67}, \cite{Pa-Pa:09}, \cite{Y} and references therein).  In particular, in \cite{Pa-Pa:09} the convergence of NCMs of eigenvalues of $\mathcal{M}_{n,m,1}$ was proved in the case when corresponding vectors $\{Y_\alpha\}_\alpha$ are "good vectors" in the sense of the following definition.

\begin{definition}
\label{d:good} We say that
a normalized isotropic vector $\mathbf{y}\in\mathbb{R}^n$ is \textit{good,} if
  for every $n\times n$ complex matrix $H_n$ which does
not depend on $\mathbf{y}$, we have
\begin{equation}
\mathbf{Var}\{(H_n\mathbf{y},\mathbf{y})\}\leq ||H_{n}||^2\delta_{n},\quad
\delta_{n}=o(1),\;n\rightarrow\infty,  \label{good}
\end{equation}
where $||H_{n}||$ is the Euclidean operator norm of $H_{n}$.
\end{definition}

Following the scheme of the proof proposed in \cite{Pa-Pa:09}, we show that despite the fact that the number of independent parameters, $kmn=O(n^{k+1})$ for $k\ge 2$ is much less than the number of matrix entries, $n^{2k}$,
the limiting distribution of eigenvalues still obeys the Marchenko-Pastur law. We have:

\begin{theorem}
\label{t:NCM}
Fix $k\ge 1$. Let $n$ and $m$ be positive integers satisfying (%
\ref{mn}),  let $\{\tau _{\alpha }\}_{\al}$ be real
numbers satisfying (\ref{sigma}), and let $\mathbf{y}$ be a good vector
 in the sense of Definition \ref{d:good}.
Then there exists a non-random measure $N$ of total mass $1$ such that
the NCMs $N_{n}$ of the eigenvalues of $\M$ (\ref{M}) converge weakly in probability
to $N$ as $n\rightarrow\infty$.
The  Stieltjes transform $f$ of $N$,
\begin{equation}
f(z)=\int \frac{N(d\lambda )}{\lambda -z},\;\Im z\neq 0,  \label{f}
\end{equation}%
  is the  unique solution of the functional equation
\begin{equation}
zf(z)=c-1-c\int(1+\tau f(z))^{-1}\sigma(d\tau)  \label{MPE}
\end{equation}
 in the class of  analytic in $\mathbb{C\setminus \mathbb{%
\ R}}$ functions  such that $\Im f(z)\Im z\geq 0,\;\Im z\neq 0. $
\end{theorem}

We use the notation $\int$ for the integrals over $\R$. Note that in \cite{T:16} there was proved a version of this statement for a deformed version of $M_{n,m,2}$.

It follows from  Theorem {\ref{t:NCM}} that if
\begin{equation}
\mathcal{N}_n[\varphi]=\sum_{j=1}^{n^k} \varphi(\lambda^{(n)}_{j})  \label{Nn}
\end{equation}
is the \emph{linear eigenvalue statistic} of $\M$  corresponding to a
bounded continuous \emph{test function} $\varphi: \mathbb{R} \to
\mathbb{C}$, then we have in probability
\begin{equation}
\lim_{n\rightarrow \infty}n^{-k}\mathcal{N}_n[\varphi]=\int \varphi(\lambda)dN(\lambda).
\label{Nnl}
\end{equation}
This can be viewed as an analog of the Law of Large Numbers in  probability
theory for (\ref{Nn}). Since the limit is non-random, the next natural step is to investigate the fluctuations  of $\mathcal{N}_n[\varphi]$.
This corresponds to the question of validity of the Central Limit Theorem (CLT). The main goal of this paper is to prove the CLT for the
 linear eigenvalue statistics of the tensor version of the sample covariance matrix $\mathcal{M}_{n,m,2}$ defined in (\ref{M}).

There is a considerable number of papers on the CLT for  linear eigenvalue
statistics of sample covariance matrices $\mathcal{M}_{n,m,1}$ (\ref{MT}), where all entries of the matrix ${B}_{n,m,1}$
are independent (see \cite{Ba-Si:04,Ba-W-Zh:10,Ba-Me:13,Cab:01,Gi:01,Ly-Pa:08,Me-Pe:16,Na-Yao:13,P-Zh:08,Sh:11} and references therein).
 Less is known in the case where the components of vector $\mathbf{y}$  are dependent. In \cite{GLPP:13},
 the CLT was proved for  linear statistics of eigenvalues of $\mathcal{M}_{n,m,1}$, corresponding to some special class of isotropic vectors defined below.
\begin{definition}
\label{d:unc} The distribution of a random vector $\mathbf{y}\in\mathbb{R}^n$
is called \textit{unconditional} if its components $\{y_j\}_{j=1}^n$ have
the same joint distribution as $\{\pm y_j\}_{j=1}^n$ for any choice of signs.
\end{definition}

\begin{definition}
\label{d:very} We say that
  normalized isotropic vectors $\mathbf{y}\in\mathbb{R}^n$, $n\in \N$, are \textit{very good} if they have unconditional
distributions, their mixed moments up to the fourth order do not depend on $i, j, n$, there exist $n$-independent $a,b\in\mathbb{R}$
such that as $n\rightarrow\infty$,
\begin{align}
& a_{2,2}:=\mathbf{E}\{y_{ i}^2y_{ j}^2\}=n^{-2}+an^{-3}+O(n^{-4}),\quad
i\neq j,  \label{a22} \\
&\kappa_{4}:=\mathbf{E}\{y_{ j}^4\}-3a_{2,2}=bn^{-2}+O(n^{-3}),\quad
  \label{k4}
\end{align}
and for every $n\times n$ complex matrix $H_n$ which does not depend on $\mathbf{y}$,
\begin{equation}
\mathbf{E}\{|(H_n\mathbf{y},\mathbf{y})^{\circ}|^4\}\leq C||H_{n}||^{4}%
n^{-2}.  \label{m4Ayy}
\end{equation}
\end{definition}

Here and in what follows we use the notation $\xi^\circ=\xi-\E\{\xi\}$.

\medskip

An important step in proving the CLT for  linear eigenvalue statistics is the asymptotic analysis of their
variances $\mathbf{Var}\{\mathcal{N}_n[\varphi]\} := \mathbf{E}\{|\mathcal{N}^\circ_n[\varphi]|^2\}$,
in particular, the proof of the bound
\begin{equation}
\mathbf{Var}\{\mathcal{N}_n[\varphi]\}\leq C_n||\varphi||^2_\mathcal{H},\label{VarN}
\end{equation}
where $||...||_\mathcal{H}$ is a functional norm and $C_n$ depends only on
$n$. This bound determines the
normalization factor in front of $\mathcal{N}^\circ_n[\varphi]$ and the class $\mathcal{H}$ of the
test functions for which the CLT, if any, is valid.
It appears that for many random matrices normalized so that there
exists a limit of their NCMs, in particular for sample covariance
matrices $\mathcal{M}_{n,m,1}$, the variance of the linear eigenvalue statistic corresponding to a smooth enough test function does not grow with $n$,
and the CLT is valid
for $\mathcal{N}^\circ_n[\varphi]$ itself without any $n$-dependent normalization factor in front.
Consider the test functions $\varphi:\mathbb{R}\rightarrow\mathbb{R}$ from the Sobolev space $\HH_s$, possessing the norm
\begin{equation}
||\varphi ||_{s}^{2}=\int (1+|t|)^{2s}|\widehat{\varphi }(t)|^{2}dt,\quad
\widehat{\varphi }(t)=\int e^{it\theta}\varphi (\theta)d\theta.  \label{Hs}
\end{equation}
The following statement was proved in \cite{GLPP:13} (see Theorem 1.8 and Remark 1.11):

\begin{theorem}
\label{t:clt} Let  $m$ and $n$ be positive integers satisfying (%
\ref{mn}) with $k=1$,  let $\{\tau _{\alpha }\}_{\alpha =1}^{m}$ be a collection of real
 numbers satisfying (\ref{sigma}) and
\begin{equation}
\sup_m\int \tau^4 d\sigma_m(\tau)<\infty,  \label{m4}
\end{equation}%
and let $\mathbf{y}$ be a very good vector in the sence of Definition \ref{d:very}. Consider matrix $\mathcal{M}_{n,m,1}(\mathbf{y})$  (\ref{M}) and
the linear statistic of its
eigenvalues $\mathcal{N}_{n}[\varphi ]$ (\ref{Nn})  corresponding to a test function $\varphi \in \HH_{s}$,  $s >2$.
 Then $\{\mathcal{N}_{n}^{\circ }[\varphi ]\}_n$ converges in distribution to a Gaussian random variable with zero
mean and the variance $V[\varphi ]=\lim_{\eta\downarrow 0}V_\eta[\varphi ]$, where
\begin{align*}
&V_\eta[\varphi ]=
\frac{1}{2\pi ^{2}}\int \int \Re \big[L(z_{1},{z_{2}}%
)-L(z_{1},\overline{z_{2}})\big](\varphi (\lambda _{1})-\varphi (\lambda
_{2}))^{2}d\lambda _{1}d\lambda _{2}
\\
 &\quad\quad\quad+ \frac{(a+b)c}{\pi ^{2}}\int\tau ^{2}\bigg(%
\Im \int \frac{f^{\prime }(z_{1})}{(1+\tau f(z_{1}))^{2}}\varphi
(\lambda _{1})d\lambda _{1}\bigg )^{2}d\sigma (\tau ),
\\
&L(z_{1},z_{2})=\frac{{\partial ^{2}}}{\partial z_{1}\partial z_{2}}\log \frac{\Delta f}{\Delta z},
\end{align*}
$z_{1,2}=\lambda _{1,2}+i\eta$, $\Delta f=f(z_{1})-f(z_{2})$, $\Delta
z=z_{1}-z_{2}$,
and $f$ given by (\ref{MPE}).
\end{theorem}

Here we prove an analog of Theorem \ref{t:clt} in the case $k=2$. We start with establishing  an analog of (\ref{VarN}) in general case $k\ge 1$:

\begin{lemma}
\label{l:apriory} Let $\{\tau
_{\alpha }\}_{\alpha}$ be a collection of real numbers satisfying (\ref{sigma}) and (\ref{m4}), and let $\mathbf{y}$ be a normalized isotropic vector having
an unconditional distribution, such that
\begin{align}
& a_{2,2}=n^{-2}+O(n^{-3}),\quad   \kappa_{4}=O(n^{-2}).\label{a220}
\end{align}
Consider the corresponding matrix $\M$  (\ref{M}) and a linear statistic of its eigenvalues $\mathcal{N}_{n}[\varphi ]$. Then
for every $\varphi \in \HH_{s}$,  $s>5/2$, and for all sufficiently large $m$ and $n$, we have
\begin{equation}
\mathbf{Var}\{\mathcal{N}_{n}[\varphi ]\}\leq C n^{k-1}||\varphi ||_{s }^{2},
\label{apriory}
\end{equation}%
where $C$ does not depend on $n$ and $\varphi$.
\end{lemma}

 It follows from Lemma \ref{l:apriory} that in order to prove the CLT (if any) for  linear eigenvalue statistics of $\M$, one needs to
normalize them by $n^{-(k-1)/2}$.
 To formulate our main result we need more definitions.
\begin{definition}
\label{d:inv} We say that the distribution of a random vector $\mathbf{y}\in\mathbb{R}^n$
is \textit{permutationally invariant} (or \textit{exchangeable}) if it is invariant with respect to the permutations of entries of $\mathbf{y}$.
\end{definition}

\begin{definition}
\label{d:clt} We say that
 normalized isotropic vectors $\mathbf{y}\in\mathbb{R}^n$, $n\in \N$, are the \textit{CLT-vectors} if they have unconditional permutationally invariant
distributions and  satisfy the following conditions:

 (i) their fourth moments satisfy  (\ref{a22}) -- (\ref{k4}),

 (ii) their sixth moments satisfy conditions
\begin{align}
& a_{2,2,2}:=\mathbf{E}\{y_{ i}^2y_{ j}^2y_{ k}^2\}=n^{-3}+O(n^{-4}),\label{m6}
\\
&a_{2,4}:=\mathbf{E}\{y_{ i}^2y_{ j}^4\}=O(n^{-3}),\quad
a_{6}:=\mathbf{E}\{y_{ i}^6\}=O(n^{-3}),\notag
\end{align}

(iii) for every $n\times n$  matrix $H_n$ which does not depend on $\mathbf{y}$,
\begin{equation}
\mathbf{E}\{|(H_n\mathbf{y},\mathbf{y})^{\circ}|^6\}\leq C||H_{n}||^{6}n^{-3}.  \label{Ayy6}
\end{equation}
\end{definition}

It can be shown that a vector of the form
$\mathbf{y}=\mathbf{x}/n^{1/2}$, where $\mathbf{x}$ has i.i.d.
components with even distribution and bounded twelfth moment is a CLT-vector as well as a vector
uniformly distributed on the unit ball in $\mathbb{R}^n$ or a properly normalized vector uniformly distributed on the unit ball
$B_p^n=\big\{\mathbf{x}\in\mathbb{R}^n:\; \sum_{j=1}^n|x_j|^p\leq 1\big\}$ in $l_p^n$ (for $k=1$, see \cite{GLPP:13} Section 2).

 The main result of the present paper is:

\begin{theorem}
\label{t:main} Let  $m$ and $n$ be positive integers satisfying (%
\ref{mn}) with $k=2$,  and let $\{\tau _{\alpha }\}_{\alpha =1}^{m}$
 be a set of real numbers uniformly bounded in $\al$ and $m$ and satisfying (\ref{sigma}).
 Consider  matrices $\mathcal{M}_{n,m,2}(\mathbf{y})$ (\ref{M}) corresponding to CLT-vectors $\y\in\R^n$.
If $\mathcal{N}_{n}[\varphi ]$ are the linear statistics of their
eigenvalues (\ref{Nn})  corresponding to a test function $\varphi \in \HH_{s}$,
$s >5/2$, then $\{n^{-1/2}\mathcal{N}_{n}^{\circ }[\varphi ]\}_n$ converges in distribution to a Gaussian random variable with zero
mean and the variance $V[\varphi ]=\lim_{\eta\downarrow 0}V_\eta[\varphi ]$, where
\begin{align}
V_\eta[\varphi]=\frac{2(a+b+2)c}{\pi^2}\int \tau^2\bigg(%
\Im \int \frac{f^{\prime }(\lambda+i\eta)}{(1+\tau f(\lambda+i\eta))^{2}}\varphi
(\lambda )d\lambda \bigg )^{2}d\sigma (\tau ),
\label{Var}
\end{align}
and $f$ given by (\ref{MPE}).
\end{theorem}

\begin{remark} {\it (i)} In particular, if $\tau _1=\cdots= \tau _m= 1$,
then
\begin{align*}
V[\varphi ]=\frac{(a+b+2)}{2c\pi ^{2}}\left( \int_{a_{-}}^{a_{+}}\varphi (\mu )\frac{\mu
-a_{m}}{\sqrt{(a_{+}-\mu )(\mu -a_{-})}}d\mu \right) ^{2},
\end{align*}%
where  $a_{\pm }=(1\pm \sqrt{c})^{2}$ and $a_{m}=1+c$.

{\it (ii)}
We can replace the condition of the uniform boundedness of $\tau_\al$ with the condition of uniform boundedness
 of eighth moments of the normalized counting measures $\sigma_n$, or take $\{\tau_\al\}_\al$ being real random variables independent of $\mathbf{y}$ with common probability law $\sigma$ having finite eighth moment. In general, it is clear from (\ref{Var}) that it should be enough to have second moments of $\sigma_n$ being uniformly bounded in $n$.

  {\it (iii)} If in (\ref{Var}) $a+b+2=0$, then to prove the CLT one needs to renormalize  linear eigenvalue statistics. In particular, it can be shown that if $\y$ in the definition of $\MM(\y)$ is uniformly distributed on the unit sphere in $\R^n$, then $a+b+2=0$ and under additional assumption $m/n=c+O(n^{-1})$ the variance of the linear eigenvalue statistic corresponding to a smooth enough test function is of the order $O(n^{k-2})$ (cf (\ref{apriory})).
\end{remark}

The paper is organized as follows. Section \ref{s:aux} contains some  known facts and auxiliary results.
In Section \ref{s:NCM}, we prove Theorem \ref{t:NCM} on the convergence of the NCMs of eigenvalues of $\MM$. Sections \ref{s:variance} and \ref{s:crocodile} present some asymptotic properties of bilinear forms $(HY,Y)$, where $Y$ is given by (\ref{Y}) and $H$ does not depend on $Y$. In Section \ref{s:apriory}, we prove Lemma \ref{l:apriory}. In Section \ref{s:covariance}, the limit expression for the covariance of the resolvent traces is found. Section \ref{s:proof} contains the proof of the main result, Theorem \ref{t:main}.

\section{Notations}
\label{s:notations}

Let $I$ be the $n^k\times n^k$ identity matrix. For $z\in \C$, $\Im z\neq 0$, let $G(z)=(\M-zI)^{-1}$ be the resolvent of $\M$, and
\begin{align*}
&\gamma_n(z)=\Tr G(z)=\sum_\jj G_{\jj\,\jj}(z),
\\
&g_n(z)=n^{-k}\gamma_n(z),\quad f_n(z)=\E\{g_n(z)\}.
 \end{align*}
Here and in what follows
$$
\sum_{\jj}=\sum_{j_1,...,j_k},\quad\sum_{j}=\sum_{j=1}^n,\quad \text{and} \quad \sum_{\al}=\sum_{\al=1}^m,
$$
so that for the non-bold Latin and Greek indices the summations are from $1$ to $n$ and from $1$ to $m$, respectively.
 For $\al\in[m]$, let
 \begin{align}
&\M^\alpha=\M\big|_{\tau_\alpha=0}=\M-\tau_{\alpha}\Y_{\alpha}\Y_\al^T,\quad G^\al(z)=(\M^\al-zI)^{-1},\label{Ma}
\\
&\gamma_n^\alpha =\Tr\, G^\alpha,\quad g_n^\al=n^{-k}\gamma_n^\alpha, \quad f_n^\al=\E\{g_n^\al\}.\notag
 \end{align}
 Thus the upper index $\al$ indicates that the corresponding function does not depend on $\Y_\al$.  We use the notations $\E_{\al}\{...\}$ and $(...)^\circ_\al $ for the averaging and the centering with respect to $\Y_\al$, so that $(\xi)^\circ_\al=\xi-\E_{\al}\{\xi\}$.

  In what follows we also need functions (see (\ref{AB}) below)
  \begin{align*}
&A_{\alpha }=A_{\alpha }(z) :=1+\tau_\alpha (G^\alpha Y_{\alpha },Y_{\alpha })
\quad\text{and}\quad B_{\alpha }=B_{\alpha }(z):=\tau_\alpha ((G^{\alpha})^2 Y_{\alpha },Y_{\alpha }).
\end{align*}

Writing $O(n^{-p})$ or $o(n^{-p})$ we suppose that $n\rightarrow\infty$ and that the coefficients in the corresponding relations are uniformly bounded in $\{\tau_\alpha\}_\alpha$, $n\in\N$, and $z\in K$.
We use the notation $K$ for any compact set in $\C\setminus \R$.

Given matrix $H$,  $||H||$ and $||H||_{HS}$ are the Euclidean operator norm and the Hilbert-Schmidt norm, respectively. We use $C$ for any absolute constant which can vary from place to place.

\section{Some facts and auxiliary results}
\label{s:aux}

We need the following bound for the martingales moments, obtained in \cite{Dh-Co:68}:
\begin{proposition}
\label{p:mart}
Let $\{S_m\}_{m\ge 1}$ be a martingale, i.e. $\forall m$,  $\E\{S_{m+1}\,\vert\,S_1,...,S_m\}=S_m$ and $\E\{|S_m|\}<\infty$. Let $S_0=0$. Then for every $\nu\ge 2$, there exists an absolute constant $C_\nu$ such that for all $m=1,2...$
\begin{align}
\label{mart}
\E\{|S_m|^\nu\}\le C_\nu m^{\nu/2-1}\sum_{j=1}^m \E\{|S_j-S_{j-1}|^\nu\}.
\end{align}
\end{proposition}

\begin{lemma}
\label{l:mart2}
Let $\{\xi_\al\}_\al$  be independent random variables assuming values in $\R^{n_\al}$ and having probability laws
$P_\al$, $\al\in[m]$, and let $\Phi: \R^{n_1}\times...\times \R^{n_m}\rightarrow\C$ be a Borel measurable function. Then for every $\nu\ge 2$, there exists an absolute constant $C_\nu$ such that for all $m=1,2...$
\begin{align}
\label{mart2}
\E\{|\Phi-\E\{\Phi\}|^\nu\}\le C_\nu m^{\nu/2-1}\sum_{\al=1}^m \E\{|(\Phi)^\circ_\al|^\nu\},
\end{align}
where $(\Phi)^\circ_\al=\Phi-\E_\al\{\Phi\}$, and $\E_\al$ is the averaging with respect to $\xi_\al$.
\end{lemma}
\begin{proof}
This simple statement is hidden in the proof of Proposition 1 in [25]. We give its proof for the sake of completeness. For $\al\in[m]$, denote  $\E_{\ge \al}=\E_{\al}...\E_m$.  Applying Proposition \ref{mart} with $S_0=0$, $S_\al=\E_{\ge \al+1}\{\Phi\}-\E \{\Phi\}$, $S_m=\Phi-\E \{\Phi\}$, we get
\begin{align*}
\E\{|\Phi-\E \{\Phi\}|^\nu\}\le C_\nu m^{\nu/2-1}\sum_{\al=1}^m \E\{|\E_{\ge \al+1}\{\Phi\}-\E_{\ge \al}\{ \Phi\}|^\nu\}.
\end{align*}
By the H$\ddot{\text{o}}$lder inequality
$$
|\E_{\ge \al+1}\{\Phi\}-\E_{\ge \al}\{ \Phi\}|^\nu=|\E_{\ge \al+1}\{(\Phi)^\circ_{\al}\}|^\nu\le \E_{\ge \al+1}\{|(\Phi)^\circ_{\al}|^\nu\},
$$
which implies (\ref{mart2}).
\end{proof}

\begin{lemma}
\label{l:Ayyl}
 Fix $\ell\ge 2$ and $k\ge 2$. Let $\mathbf{y}\in \mathbb{R}^n$
be a normalized isotropic random vector  (\ref{iso}) such that
 for every $n\times n$ complex matrix $H$ which does
not depend on $\mathbf{y}$, we have
\begin{equation}
\E\{|(H\mathbf{y},\mathbf{y})^\circ|^\ell\}\leq ||H||^\ell\delta_{n},\quad
\delta_{n}=o(1),\;n\rightarrow\infty.  \label{Hyyl}
\end{equation}
Then there exists an absolute constant $C_\ell$ such that for every $n^k\times n^k$ complex matrix $\mathcal{H}$ which does
not depend on $\mathbf{y}$, we have
\begin{equation}
\E\{|(\mathcal{H}\Y,\Y)^\circ|^\ell\}\leq C_\ell k^{\ell/2}||\mathcal{H}||^\ell\delta_{n},  \label{HYYl}
\end{equation}
where $\Y=\mathbf{y}^{(1)}\otimes...\otimes\mathbf{y}^{(k)}$, and $\mathbf{y}^{(j)}$, $j\in[k]$, are i.i.d. copies of
$\mathbf{y}$.
\end{lemma}

\begin{proof}
It follows  from (\ref{mart2}) that
\begin{align}
\label{bn0l}
\E\{|(\mathcal{H}\Y,\Y)^\circ|^\ell\}\le C_\ell k^{\ell/2-1}\sum_{j=1}^k \E\{|(\mathcal{H}\Y,\Y)^\circ_j|^\ell\},
\end{align}
where
$\xi^\circ_{j}=\xi-\E_{j}\{\xi\}$ and $\E_j$ is the averaging w.r.t. $y^{(j)}$.
We have
\begin{equation*}
(\mathcal{H}\Y,\Y)=\sum_{\pp,\,\q}\mathcal{H}_{\pp,\,\q}\Y_\pp\Y_\q=(H^{(j)}\mathbf{y}^{(j)},\mathbf{y}^{(j)}),
\end{equation*}
where $H^{(j)}$ is an $n\times n$ matrix with the entries
$$
(H^{(j)})_{st}=\sum_{\pp,\,\q}\mathcal{H}_{\pp,\,\q}\, \delta_{p_js}\delta_{q_jt} \,\,
y^{(1)}_{p_1}...y^{(j-1)}_{p_{j-1}}y^{(j+1)}_{p_{j+1}}...y^{(k)}_{p_k}\,\,
y^{(1)}_{q_1}...y^{(j-1)}_{q_{j-1}}y^{(j+1)}_{q_{j+1}}...y^{(k)}_{q_k}.
$$
This and (\ref{Hyyl}) yield
\begin{align*}
 \E_j\{|(\mathcal{H}\Y,\Y)^\circ_j|^\ell\}=\E_j\{|(H^{(j)}\mathbf{y}^{(j)},\mathbf{y}^{(j)})^\circ|^\ell\}\le  ||H^{(j)}||^\ell\delta_{n}.
\end{align*}
We have
$$
||H^{(j)}||\le ||\mathcal{H}||\prod _{i\neq j}||\mathbf{y}^{(i)}||^2.
$$
For $i\in[k]$, since by (\ref{iso}) $\E\{||y^{(i)}||\}=1$,  we have by (\ref{Hyyl}) $\E\{||y^{(i)}||^{2\ell}\}\le C$.
 Hence
\begin{align*}
 \E_j\{|(\mathcal{H}\Y,\Y)^\circ_j|^\ell\}
 \le ||\mathcal{H}||^{\ell}\prod _{i\neq j}\E\{||\mathbf{y}^{(i)}||^{2\ell}\}\delta_{n}\le C||\mathcal{H}||^\ell\delta_{n}.
\end{align*}
This and (\ref{bn0l}) lead to (\ref{HYYl}), which completes the proof of the lemma.
\end{proof}

\medskip

 The following statement was proved in \cite{Pa-Pa:09}.
\begin{proposition}
\label{p:sa}
Let $N_{n}$ be the NCM of the eigenvalues
of ${M}_n=\sum_{\alpha}\tau_\alpha {{\Y}_\alpha}{{\Y}_\alpha}^T$, where $\{\Y_{\alpha }\}_{\alpha =1}^{m}\in\R^{p}$ are i.i.d. random
vectors and $\{\tau _{\alpha }\}_{\alpha =1}^{m}$ are real numbers.  Then
\begin{align}
&\mathbf{Var}\{N_{n}(\Delta )\}\leq 4m/p^{2},\quad \forall \Delta \subset \mathbb{R}, \label{varN}
\\
&\mathbf{Var}\{g_{n}(z)\}\leq 4m/(p|\Im z|)^{2},\quad \forall z\in \mathbb{C}\setminus \mathbb{R}.\label{varg1}
\end{align}%
\end{proposition}

 Also, we will need the following simple claim:
\begin{claim}
\label{c:varprod}
If $h_1$, $h_2$ are bounded random variables, then
\begin{align}
  \label{varprod}
\V\{h_1h_2\}\le C\big(\V\{h_1\}+\V\{h_2\}\big).
\end{align}

\end{claim}

\bigskip
\section{Proof of Theorem \ref{t:NCM}}
\label{s:NCM}

 Theorem \ref{t:NCM} essentially follows from Theorem 3.3 of \cite{Pa-Pa:09}
 and Lemma \ref{l:Ayyl}, here we give a proof for the sake of completeness. In view of (\ref{varN}) with $p=n^k$, it suffices to prove that the expectations
$
\overline{N}_{n}=\mathbf{E}\{N_{n}\}
$
of the NCMs of the eigenvalues of $\mathcal{M}_n$
converge weakly to $N$. Due to the one-to-one correspondence between non-negative measures and their Stieltjes transforms
(see e.g. \cite{Ak-Gl:93}), it is enough to show that the Stieltjes transforms of $\overline{N}_{n}$,
\begin{equation*}
f_n(z)=\int\frac{\overline{N}_{n}(d\lambda)}{\lambda-z},
\end{equation*}%
 converge to the solution $f$ of (\ref{MPE}) uniformly on every compact set $K\subset \C\setminus\R$, and that
 \begin{equation}
\lim_{\eta\rightarrow\infty} \eta|f(i\eta)|=1.\label{NR}
\end{equation}
 In \cite{Pa-Pa:09}, it is proved that the solution of (\ref{MPE}) satisfies (\ref{NR}), so it is enough to show that
 \begin{equation}
f_n(z){\underset{n\rightarrow\infty} \rightrightarrows} f(z),\quad z\in K,\label{fnf}
\end{equation}%
where we use the double arrow notation for the uniform convergence.
Assume first that all $\tau_\al$ are bounded:
\begin{equation}
\label{L}
\forall m\,\, \forall\al\in[m]\quad|\tau_\al|\le L.
\end{equation}
 Since $\M-\M^{\alpha}=\tau_{\alpha}\Y_{\alpha}\Y_\al^T$, the rank one perturbation formula
 \begin{equation}
 G-G^\alpha=-\frac{\tau_\alpha G^\alpha \Y_{\alpha}\Y_\al^T G^\alpha}{1+\tau_\alpha (G^\alpha Y_{\alpha },Y_{\alpha })}\label{G-G}
\end{equation}
implies  that
 \begin{equation}
 \gamma _{n}-\gamma_{n}^\alpha =-\frac{\tau_\alpha ((G^{\alpha})^2 Y_{\alpha },Y_{\alpha })}{1+\tau_\alpha (G^\alpha Y_{\alpha },Y_{\alpha })}=-\frac{B_{\al}}{A_{\al}}.\label{AB}
\end{equation}
It follows from the spectral theorem for the real symmetric matrices that there
exists a non-negative measure $m^{\alpha}$ such that
\begin{equation}
 (G^\alpha Y_{\alpha },Y_{\alpha })= \int\frac{m^\alpha(d\lambda)}{\lambda-z},\quad ((G^\alpha)^2 Y_{\alpha },Y_{\alpha })= \int\frac{m^\alpha(d\lambda)}{(\lambda-z)^2}.
\label{gya}
\end{equation}
This  yields
\begin{equation*}
|A_\alpha| \ge |\Im A_\alpha|
=|\tau_\alpha| |\Im z| \int\frac{ m^\alpha(d\lambda)}{|\lambda-z|^{2}},\quad |B_\alpha| \le
|\tau_\alpha| \int\frac{ m^\alpha(d\lambda)}{|\lambda-z|^{2}},
\end{equation*}
implying that
\begin{align}
|B_{\al}/A_{\al}|\le 1/|\Im z|.\label{B/A<}
\end{align}
It also follows from (\ref{G-G}) that
$
A^{-1}_{\alpha}=1-\tau_{\alpha}(G Y_{\alpha },Y_{\alpha}).
$
Hence,
\begin{equation}
|A^{-1}_{\alpha}|\le 1+|\tau_{\alpha}|\cdot|| Y_{\alpha }||^2/|\Im z|,
\label{A>}
\end{equation}
where we use $||G||\le|\Im z|^{-1}$. Let us show that
\begin{equation}
{|\mathbf{E}_\al\{A_{\al}\}|}^{-1},\,{|\mathbf{E}\{A_{\al}\}|}^{-1} \le  4(1+|\tau_\al|/|\Im z|).
\label{EA>}
\end{equation}
It follows from (\ref{iso}) that
\begin{align}
\E_\al\{A_\al\}=1+\tau_\al g_n^\al(z),\,\,\E\{A_\al\}=1+\tau_\al f_n^\al(z).\label{EA}
\end{align}
Consider $\E_\al\{A_\al\}$. By the spectral theorem for the real symmetric matrices,
\begin{equation*}
 \E_\al\{A_\al\}=1+\tau_\al n^{-k} \int\frac{\NN_n^\alpha(d\lambda)}{\lambda-z},
\end{equation*}
 where $\NN_{n}^{\alpha}$ is the counting measure of the eigenvalues of $\M^{\al}$.
 For every $\eta \in\R\setminus\{0\}$, consider
 $$
 E_{\eta}=\Big\{z=\mu+i\eta\,:\,\Big| n^{-k} \int\frac{\NN_n^\alpha(d\lambda)}{\lambda-z}\Big|\le \frac{1}{2|\tau_\al|}\Big\}.
 $$
Clearly, for $z\in E_{\eta}$, $|\mathbf{E}_\al\{A_{\al}\}|\ge 1/2$. If $z=\mu +i\eta\notin E_{\eta}$, then
 $$
 \frac{1}{2|\tau_\al|}<\Big| n^{-k} \int\frac{\NN_n^\alpha(d\lambda)}{\lambda-z}\Big|
 \le\Big( n^{-k} \int\frac{\NN_n^\alpha(d\lambda)}{|\lambda-z|^2}\Big)^{1/2},
 $$
so that
$$
| \E_\al\{A_\al\}|\ge| \Im \E_\al\{A_\al\}|=|\tau_\al||\eta|n^{-k} \int\frac{\NN_n^\alpha(d\lambda)}{|\lambda-z|^2}\ge \frac{|\eta|}{4|\tau_\al|}.
 $$
 This leads to (\ref{EA>}) for $\E_\al\{A_\al\}$. Replacing in our argument $\NN_n^\alpha$ with $\overline{\NN}_n^\alpha$, we get (\ref{EA>}) for $\E\{A_\al\}$.

It follows from the resolvent identity and (\ref{G-G}) that
\begin{align}
zg_n(z)=-1+n^{-k}\Tr \M G=\big(-1+{m}n^{-k}\big)-n^{-k}\sum_\al{A^{-1}_\al}.\label{gn=}
\end{align}
This and the identity
\begin{equation}
\frac{1}{A_\al}=\frac{1}{\mathbf{E}\{A_\al\}}-\frac{A_\al^\circ}{A_\al\mathbf{E}\{A_\al\}}
\label{1/A}
\end{equation}
lead to
\begin{align*}
&zf_n(z)=\big(-1+n^{-k}\big)-n^{-k}\sum_\al{\E\{A_\al\}}^{-1}+r_n(z),
\\
&r_n(z)=n^{-k}\sum_\al\frac{1}{\E\{A_\al\}}\E\Big\{\frac{A_\al^\circ}{A_\al}\Big\}.
\end{align*}
It follows from the Schwarz inequality that
$$
|\E\{{A_\al^\circ}A_\al^{-1}\}|\le\E\{|{A_\al^\circ}|^2\}^{1/2}\E\{|A^{-2}_\al|\}^{1/2}.
$$
Note that  since $\E\{||Y_\al||=1\}$, we have by (\ref{good}) $\E\{||Y_\al||^4\}\le C$.
This and (\ref{A>}) imply that $\E\{|A^{-2}_\al|\}$ is uniformly bounded in $|\tau_\al|\le L$ and $z\in K$. We also have
\begin{equation}
A_{\alpha}^\circ=(A_{\alpha })_\alpha^\circ+\tau_{\alpha}(g_n^{\alpha})^{\circ}=
\tau_{\alpha}\big[(G^\al \Y_\al,\Y_\al)_\alpha^\circ+(g_n^{\alpha})^{\circ}\big],\label{A0Aa0}
\end{equation}
hence
\begin{equation*}
\E\{|{A_\al^\circ}|^2\}=\tau_\al^2 \Big(\E\{\E_\al\{|(G^\al \Y_\al,\Y_\al)_\alpha^\circ|^2\}\}+\E\{|(g_n^{\alpha})^{\circ}|^2\}\Big).
\end{equation*}
By (\ref{mn}) and (\ref{varg1}) with $p=n^k$, $\V\{g_n^\al\}\le Cn^{-k}|\Im z|^{-2}$. It follows from (\ref{good}) and Lemma \ref{l:Ayyl} with $\mathcal{H}=G^\al$ and $\ell=2$ that
 \begin{equation*}
\E_\al\{|(G^\al \Y_\al,\Y_\al)_\alpha^\circ|^2\}\le C_2k ||G^\al||^{2}\delta_n\le C_2k |\Im z|^{-2}\delta_n.
\end{equation*}
Thus, $\E\{|{A_\al^\circ}|^2\}\le CL^2|\Im z|^{-2}(k\delta_n+n^{-k})$.
This and (\ref{EA>}) yield
\begin{equation}
\label{rnk}
|r_n|\le C(k{\delta_n}+n^{-k})^{1/2}.
\end{equation}
uniformly in $|\tau_\al|\le L$ and $z\in K$. Hence
\begin{equation}
\label{eqfn}
zf_n(z)=(-1+{m}n^{-k})-n^{-k}\sum_\al(1+\tau_\al f_n^\al(z))^{-1}+o(1).
\end{equation}
It follows from (\ref{AB}) and (\ref{B/A<}) that
\begin{equation}
\label{ffa}
|f_n(z)-f_n^\al(z)|\le n^{-k}|\Im z|^{-1}.
\end{equation}
This and (\ref{EA>}) implies that $|1+\tau_\al f_n(z)|^{-1}$ is uniformly bounded in $|\tau_\al|\le L$ and $z\in K$.
Hence, in (\ref{eqfn}) we can replace $f_n^\al$  with $f_n$ (the corresponding error term is of the order $O(n^{-k})$) and pass to the limit as
 $n\rightarrow\infty$.
Taking into account (\ref{sigma}) we get  that the limit of every convergent subsequence of $\{f_n(z)\}_n$ satisfies (\ref{MPE}).
This  finishes the proof of the theorem under assumption (\ref{L}).

Consider now the general case  and take any sequence $\{\sigma_n\}=\{\sigma_{m(n)}\}$ satisfying (\ref{sigma}). For any $L>0$, introduce the truncated random variables
\begin{equation*}
\tau _{\alpha }^{L}=\left\{
\begin{array}{cc}
\tau _{\alpha }, & |\tau _{\alpha }|<L, \\
0, & \text{otherwise}. 
\end{array}%
\right.  
\end{equation*}%
 Denote $\M^L=\sum_{\alpha =1}^{m}\tau _{\alpha }^{L}\Y_{\alpha}\Y_{\alpha}^T.$ Then
\begin{equation*}
\mathrm{rank}(\M-\M^L)\leq \mathrm{Card}\{\alpha \in [m]:\,|\tau _{\alpha }|\geq L\}.
\end{equation*}%
 Take any sequence $\{L_i\}_i$ which does not contain atoms of $\sigma$ and tends to infinity as $i\rightarrow
\infty $. If $N_{n}^{L_i}$ is the NCM of the eigenvalues of $%
\M^{L_i}$ and $\overline{N}_{n}^{L_i}$ is its expectation, then the
mini-max principle implies that for any interval $\Delta \subset \mathbb{R}$:%
\begin{equation*}
|\overline{N}_{n}(\Delta )-\overline{N}_{n}^{L_i}(\Delta )|\leq \int_{|\tau|\ge L_i}\sigma_n(d\tau).
\end{equation*}%
We have
\begin{equation*}
 \int_{|\tau|\ge L_i}\sigma_n(d\tau)=
\int_{|\tau|\ge L_i}(\sigma_n-\sigma)(d\tau)+\int_{|\tau|\ge L_i}\sigma(d\tau),
\end{equation*}%
where by (\ref{sigma}) the first term on the r.h.s. tends to zero as $n\rightarrow
\infty $. Hence,
$$
\lim_{L_i\rightarrow\infty}\lim_{n\rightarrow\infty}\int_{|\tau|\ge L_i}\sigma_n(d\tau)=0.
$$
 Thus if $f$ and $f^{L_{i}}$ are  the Stieltjes transforms of $\overline{N}$ and $\lim_{n\rightarrow\infty}\overline{N}_{n}^{L_{i}}$, then
 $$
 f(z)=\lim_{i\rightarrow\infty} f^{L_{i}}(z)
 $$
  uniformly on $K$.
It follows from the first part of the proof that
\begin{equation}
zf^{L_{i}}(z)=-1-c_{L_{i}}f^{L_{i}}(z)\int_{-L_i}^{L_{i}}\tau(1+\tau f^{L_{i}}(z))^{-1}\sigma(d\tau),\label{MPEL}
\end{equation}
where $c_{L_{i}}=c\sigma[-L_i,L_{i}]\rightarrow c$ as $L_{i}\rightarrow\infty$. Since $%
N(\mathbb{R})=1$, there exists $C>0$, such that%
\begin{equation*}
\min_{z\in K}|\Im f(z)|=C>0.
\end{equation*}%
Hence we have for all sufficiently big $L_{i}$:%
\begin{equation*}
\min_{z\in K}|\Im f^{L_{i}}(z)|=C/2>0.
\end{equation*}%
Thus $|\tau /(1+\tau f^{L_{i}}(z))|\leq |\Im f^{L_{i}}(z)|^{-1}\leq
2/C<\infty ,\;z\in K$. This allows us to pass to the limit $L_i\rightarrow
\infty $ in (\ref{MPEL}) and to obtain (\ref{MPE}) for $f$, which completes the proof of the theorem.
 \qed

 \begin{remark}
 It follows from the proof that in the model we can take $k$ depending on $n$ such that
 $$
 k\rightarrow\infty\quad\text{and}\quad k\delta_n\rightarrow0
 $$
 as $n\rightarrow\infty$, and the theorem remains valid (see (\ref{rnk})).
 \end{remark}

\section{Variance of bilinear forms}
\label{s:variance}

\begin{lemma}
\label{l:form}
Let $\Y$ be defined in (\ref{Y}) -- (\ref{iso}), where $\y$ has an unconditional distribution and satisfies
(\ref{a220}).
Then  for every symmetric $n^k\times n^k$ matrix $H$ which does not depend on $\y$ and whose operator norm is uniformly bounded in $n$, there is an absolute constant $C$ such that
\begin{align}
\label{Varform<}
n\V\{(H\Y,\Y)\}\le Cn^{-k}||H||_{HS}^2\le C||H||^2.
\end{align}
If additionally $\y$ satisfies (\ref{a22}) -- (\ref{k4}), then we have
\begin{align}
\label{formk}
n\V\{(H\Y,\Y)\}=&k\aaa |n^{-k}\Tr H|^2
\\
&+n^{-2k+1}\sum_{i=1}^k\sum_{\jj,\pp}\big[2H_{\jj,\,\jj(p_i)}\overline{H}_{\pp,\,\pp(j_i)}
+\bb H_{\jj,\,\jj}\overline{H}_{\pp,\,\pp}\delta_{p_ij_i}\big]+O(n^{-1}),\notag
\end{align}
where $\jj(p_i)=\{j_1,...,j_{i-1},p_i,j_{i+1},...,j_k\}$.
\end{lemma}

\begin{proof}
Since $\y$ has an unconditional distribution, we have
\begin{equation}
 \mathbf{E}\{y_{j}y_{s}y_{p}y_{q}\}=
 a_{2,2}(\delta_{js}\delta_{pq}+\delta_{jp}\delta_{sq}+\delta_{jq}
 \delta_{sp})+\kappa_{4}\delta_{js}\delta_{jp}\delta_{jq}.
\label{yyyy}
 \end{equation}
Hence,
\begin{align*}
\E\{|(H\Y,\Y)|^2\}=&\sum_{\jj,\s,\pp,\q}H_{\jj,\,\s}\overline{H}_{\pp,\,\q}
\prod_{i=1}^k\Big[ a_{2,2}\delta_{j_is_i}\delta_{p_iq_i}+w_i\Big],\notag
\end{align*}
where
$$
w_i=w_i(\jj,\s,\pp,\q)=a_{2,2}(\delta_{j_ip_i}\delta_{s_iq_i}+\delta_{j_iq_i}
 \delta_{s_ip_i})+\kappa_{4}\delta_{j_is_i}\delta_{j_ip_i}\delta_{j_iq_i}.
 $$
 For $W\subset[k]$, $W^c=[k]\setminus W$, denote
\begin{align*}
\Lambda(W,\jj,\s,\pp,\q)=
\prod_{i\in W^c}( a_{2,2}\delta_{j_is_i}\delta_{p_iq_i})\prod_{\ell\in W}w_\ell.
\end{align*}
For every fixed $W,\jj,\s$, we have
 \begin{align}
\sum_{\pp,\q}\Lambda(W,\jj,\s,\pp,\q)=O(n^{-k-|W|}).\label{sumLW}
\end{align}
Indeed,  the number of pairs for which $\Lambda(W,\jj,\s,\pp,\q)\neq 0$ does not exceed $2^{|W|}n^{k-|W|}$ (the number of choices of indices $p_i=q_i$ for $i\notin W$ equals to $n^{k-|W|}$; all other indices $p_\ell,\, q_\ell$ ($\ell\in W$) must satisfy $\{p_\ell,\, q_\ell\}=\{j_\ell,\, s_\ell\}$ and, therefore, can be chosen in at most two ways each). Since $a_{2,2}$, $w_i=O(n^{-2})$, (\ref{sumLW}) follows.

For every fixed $W$,
\begin{align}
\sum_{\jj,\s,\pp,\q}|H_{\jj,\,\s}||{H}_{\pp,\,\q}|\Lambda(W,\jj,\s,\pp,\q)&\le
\sum_{\jj,\s,\pp,\q}\big(|H_{\jj,\,\s}|^2+|H_{\pp,\,\q}|^2\big)\Lambda(W,\jj,\s,\pp,\q)/2\notag
\\
&=O(n^{-k-|W|})||H||_{HS}^2.\label{OW}
\end{align}
Since by (\ref{iso})
$
\E\{(H\Y,\Y)\}=n^{-k}\Tr H,
$
we have
\begin{align}
\label{V=}
\V\{(H\Y,\Y)\}
=\sum_{r=0}^k\sum_{|W|=r}\sum_{\jj,\s,\pp,\q}H_{\jj,\,\s}\overline{H}_{\pp,\,\q}\Lambda(W,\jj,\s,\pp,\q)-n^{-2k}|\Tr H|^2.
\end{align}
  By (\ref{a220}), the term corresponding to $W=\emptyset$, $W^c=[k]$, has the form
\begin{align*}
T_0:=\sum_{\jj,\s,\pp,\q}H_{\jj,\,\s}\overline{H}_{\pp,\,\q}\prod_{i=1}^k( a_{2,2}\delta_{j_is_i}\delta_{p_iq_i})=
a_{2,2}^k|\Tr H|^2.
\end{align*}
This and (\ref{a220}) imply that
\begin{align*}
n\big|T_0-n^{-2k}|\Tr H|^2\big|\le C n^{-k}||H||_{HS}^2,
\end{align*}
and by (\ref{a22}),
\begin{align}
n(T_0-n^{-2k}|\Tr H|^2)=ka n^{-2k}|\Tr H|^2+O(n^{-1}).\label{W0}
\end{align}
The term corresponding to  $\sum_{|W|=1}$ (i.e. $W=\{1\},...,W=\{k\}$), has the form
\begin{align*}
T_1:&=\sum_{i=1}^k\sum_{\jj,\s,\pp,\q}H_{\jj,\,\s}\overline{H}_{\pp,\,\q}\,w_i(\jj,\s,\pp,\q)
 \prod_{\ell\neq i}a_{2,2}\delta_{j_\ell s_\ell}\delta_{p_\ell q_\ell}\notag
 \\
& =\sum_{i=1}^k\sum_{\jj,\pp}\big[a_{2,2}^kH_{\jj,\,\jj(p_i)}\overline{H}_{\pp,\,\pp(j_i)}
+a_{2,2}^{k-1}\kappa_4 H_{\jj,\,\jj}\overline{H}_{\pp,\,\pp}\delta_{p_ij_i}\big],
\end{align*}
and by (\ref{a22})
\begin{align}
nT_1 =n^{-2k+1}\sum_{i=1}^k\sum_{\jj,\pp}\big[2H_{\jj,\,\jj(p_i)}\overline{H}_{\pp,\,\pp(j_i)}
+\bb H_{\jj,\,\jj}\overline{H}_{\pp,\,\pp}\delta_{p_ij_i}\big]+O(n^{-1}).\label{W1}
\end{align}
Also it follows from (\ref{OW}) that the terms corresponding to $W$: $|W|\ge2$ are less than $Cn^{-k-2}||H||_{HS}^2$.
Summarizing (\ref{V=}) -- (\ref{W1}), we get (\ref{Varform<}) and (\ref{formk}) and complete the proof of the lemma.
\end{proof}

\section{Proof of Lemma \ref{l:apriory}}
\label{s:apriory}

\begin{lemma}
\label{l:varg} Let $\{\tau
_{\alpha }\}_{\alpha}$ be a collection of real numbers satisfying (\ref{sigma}), (\ref{m4}), and let $\mathbf{y}$ be a normalized isotropic vector having
an unconditional distribution and satisfying (\ref{a220}). Consider the corresponding matrix $\M$  (\ref{M}) and the trace of its resolvent
$\gamma _{n}(z)=\Tr (\M-zI)^{-1}$. We have
\begin{equation}
\mathbf{Var}\{\gamma _{n}(z)\}\leq C n^{k-1}|\Im z|^{-6}.  \label{varg}
\end{equation}%
If additionally  $\mathbf{y}$ satisfies (\ref{m4Ayy}) and $\tau_\al$ are uniformly bounded in $\al$ and $m$, then
\begin{equation}
\mathbf{E}\{|\gamma _{n}^{\circ }(z)|^{4}\}\leq Cn^{2k-2}|\Im z|^{-12}.
\label{g4<}
\end{equation}
\end{lemma}
\begin{proof}
The proof  follows the scheme proposed in \cite{Sh:11} (see also Lemma 3.2 of \cite{GLPP:13}). For $q=1,2$, by (\ref{mart2}) we have
\begin{align}
 \mathbf{E}\{|\gamma_n^\circ|^{2q}\}&\le Cm^{q-1}\sum_{\alpha}\mathbf{E}\{|(\gamma_n)^\circ_\alpha|^{2q}\}.\label{gq}
\end{align}
Applying (\ref{AB}), (\ref{B/A<}), and (\ref{EA>}) we get
 \begin{align}
 \label{gAB}
\mathbf{E}\{|(\gamma_n)^\circ_\alpha|^{2q}\}&= \mathbf{E}\{|\gamma _{n}-\gamma^\alpha_{n}-\mathbf{E}_\alpha\{\gamma _{n}-\gamma^\alpha_{n}\}|^{2q}\}
\\
&\leq C\mathbf{E}\Big\{\Big|\frac{B_{\alpha}}{A_{\alpha}}-
\frac{\mathbf{E}_\alpha\{B_{\alpha}\}}{\mathbf{E}_\alpha\{A_{\alpha}\}}\Big|^{2q}\Big\}=
C\mathbf{E}\Big\{\Big|\frac{(B_{\alpha })^\circ_\alpha}{\mathbf{E}_\alpha\{A_{\alpha}\}}-\frac{B_{\alpha}}{A_{\alpha}}
\cdot\frac{(A_{\alpha})^\circ_\alpha}{\mathbf{E}_\alpha\{A_{\alpha}\}}
\Big|^{2q}\Big\}\notag
\\
&\le
 C(1+|\tau_\al|/|\Im z|)^{2q}\E\big\{\E_\al\{|(B_{\alpha })^\circ_\alpha|^{2q}\}+ \E_\al\{|(A_{\alpha })^\circ_\alpha|^{2q}\}/|\Im z|^{2q}\big\}.\notag
 \end{align}
Here by  (\ref{Varform<})
 \begin{align}
&n\tau_\al^{-2} \E_\al\{|(A_{\alpha })^\circ_\alpha|^{2}\}=n\E_\al\{|(G^{\alpha }Y_\al,Y_\al)^\circ_\alpha|^{2}\}\le Cn^{-k}||G^\al||^2_{HS}
\le C|\Im z|^{-2}\label{AHS}
 \end{align}
 and
  \begin{align}
&n\tau_\al^{-2} \E_\al\{|(B_{\alpha })^\circ_\alpha|^{2}\}\le Cn^{-k}
||(G^{\alpha})^2||^2_{HS}\le|\Im z|^{-4}.\label{BHS}
 \end{align}
This and (\ref{gq}) -- (\ref{gAB}) lead to (\ref{varg}). Also it follows from (\ref{m4Ayy})  and Lemma \ref{l:Ayyl} that
$$
\E_\al\{|(B_{\alpha})^\circ_\alpha|^{4}\},\,\ \E_\al\{|(A_{\alpha})^\circ_\alpha|^{4}\}/|\Im z|^{4}\le C\tau_\al^{4}|\Im z|^{-8}n^{-2},
$$
which leads to (\ref{g4<}).
\end{proof}

\medskip

\noindent {\bf Proof of Lemma \ref{l:apriory}}.
The proof of (\ref{apriory}) is based on the following inequality obtained in \cite{Sh:11}: for $\varphi\in \HH_s$ (see (\ref{Hs})),
\begin{equation*}
\mathbf{Var}\{\mathcal{N}_{n}[\varphi ]\}\leq C_{s}||\varphi
||_{s}^{2}\int_{0}^{\infty }d\eta e^{-\eta }\eta ^{2s-1}\int\mathbf{Var}\{\gamma _{n}(\mu +i\eta )\}d\mu.
\end{equation*}%
Let $z=\mu +i\eta$, $\eta>0$. It follows from (\ref{gq}) -- (\ref{BHS}) that
\begin{align*}
\V\{\gamma_n\}&\le \sum_{\alpha}\mathbf{E}\{|(\gamma_n)^\circ_\alpha|^{2}\}
\\
&\le C n^{-k-1}\sum_{\alpha}\tau_\al^{2}(1+\eta^{-2}\tau_\al^{2})\mathbf{E}\{||(G^{\alpha})^2||^2_{HS}+\eta^{-2}||G^{\alpha}||^2_{HS}\}.
\end{align*}
By the spectral theorem for the real symmetric matrices,
\begin{equation*}
\E\big\{||G^{\alpha}||^2_{HS}\big\}= \int\frac{\overline{\NN^\alpha_n}(d\lambda)}{|\lambda-z|^2}, \quad
\E\big\{||(G^{\alpha})^2||^2_{HS}\big\}=  \int\frac{\overline{\NN^\alpha_n}(d\lambda)}{|\lambda-z|^4},
\end{equation*}
 where $\overline{\NN_{n}^{\alpha}}$ is the expectation of the counting measure of the eigenvalues of $\M^{\al}$. We have
 \begin{equation*}
 n^{-k}\int \int\frac{\overline{\NN^\alpha_n}(d\lambda)}{|\lambda-z|^2}d\mu\le C\eta^{-1}, \quad  n^{-k} \int\int\frac{\overline{\NN^\alpha_n}(d\lambda)}{|\lambda-z|^4}d\mu\le C\eta^{-3}.
\end{equation*}
Summarizing, we get
\begin{equation*}
\mathbf{Var}\{\mathcal{N}_{n}[\varphi ]\}\leq Cn^{k-1}||\varphi
||_{s}^{2}\int_{0}^{\infty }d\eta e^{-\eta }\eta ^{2s-6}\le Cn^{k-1}||\varphi
||_{s}^{2}
\end{equation*}%
provided that $s>5/2$.
This finishes the proof of Lemma \ref{l:apriory}.

\medskip

\section{Case $k=2$. Some preliminary results}
\label{s:crocodile}

From now on we fix $k=2$ and consider matrices $\M=\mathcal{M}_{n,m,2}$. For every $\jj=\{j_1,j_2\}=j_1j_2$,
$$
\sum_{\jj}=\sum_{j_1,j_2},\quad \sum_{j}=\sum_{j=1}^n.
$$

In this section we establish some asymptotic properties of $A_\al$,  $(G^\al Y_\al,Y_\al)$, and their central moments.
We start with
\begin{lemma}
\label{l:aux} Under conditions of Theorem \ref{t:main},
\begin{align}
&\E_{\al}\{|(A_{\alpha })_\al^\circ|^p\}\le C(\tau_\alpha/|\Im z|)^{p}n^{-p/2},\label{EA0p<}
\\
&\E_{\al}\{|(B_{\alpha })_\al^\circ|^p\}\le C(\tau_\alpha/|\Im z|^2)^{p}n^{-p/2},  \notag
\end{align}
and
\begin{align}
&\mathbf{E}\{|A_{\alpha }^\circ|^p\},\,\mathbf{E}\{|B_{\alpha}^\circ|^p\}=O(n^{-p/2}),\quad 2\le p\le 6.\label{A0B0}
\end{align}
\end{lemma}
\begin{proof}
 Since $(A_{\alpha })_\al^\circ=\tau_\al(G^{\al}Y_\al,Y_\al)_\al^\circ$, Lemma \ref{l:Ayyl} and (\ref{Ayy6}) imply that
$$
\E_{\al}\{|(A_{\alpha })_\al^\circ|^6\}\le C(\tau_\alpha/|\Im z|)^{6}n^{-3},
$$
and by the H$\ddot{\text{o}}$lder inequality we get the first estimate in (\ref{EA0p<}). Analogously one can get the second estimate in (\ref{EA0p<}).
Also we have by (\ref{varg})
$$
\E\{|(g_n^{\al})^{\circ}|^p\}\le |\Im z|^{2-p}\E\{|(g_n^{\al})^{\circ}|^2\}=O(n^{-3}),\quad p\ge 2,
$$
which together with (\ref{A0Aa0}) and  (\ref{EA0p<}) leads to (\ref{A0B0}).
\end{proof}

Let
$$
H=H(z)=G^\al(z).
$$
 It follows from (\ref{formk}) with $k=2$ that
\begin{align}
\label{form2}
n\V\{(H\Y,\Y)\}=&2\aaa |n^{-2}\Tr H|^2
\\
&+2n^{-3}\sum_{\jj,\pp}\big[H_{\jj,\,j_1p_2}\overline{H}_{\pp,\,p_1j_2}+H_{\jj,\,p_1j_2}\overline{H}_{\pp,\,j_1p_2}\big]\notag
\\
&+\bb n^{-3}\sum_{\jj,\pp}H_{\jj,\,\jj}\overline{H}_{\pp,\,\pp}(\delta_{p_1j_1}+\delta_{p_2j_2})+O(n^{-1}).\notag
\end{align}
Consider an $n\times n$ matrix of the form
\begin{equation*}
\G=\{\G_{s,p}\}_{s,p=1}^n,\quad \G_{s,p}=\sum_{j}H_{js,\,jp}.
\end{equation*}%
Since $\G=\sum_j \G^{(j)}$, where for every $j$, $\G^{(j)}=\{H_{js,jp}\}_{s,p}$ is a block of $G^\al$, we have
\begin{equation}
\label{normG}
||\G||\le\sum_j ||\G^{(j)}||\le n||G^\al||\le n/|\Im z|.
\end{equation}
We define functions
\begin{align*}
&g_n^{(1)}(z_1,z_2):=n^{-3}\sum_{\jj,\pp}H_{\jj,\,j_1p_2}(z_1)H_{\pp,\,p_1j_2}(z_2)=n^{-3}\Tr \G(z_1)\G(z_2),\notag
\\
&g_n^{(2)}(z_1,z_2):= n^{-3}\sum_{i,s, j}H_{is,\,is}(z_1)H_{js,\,js}(z_2)= n^{-3}\sum_{s}\G_{ss}(z_1)\G_{ss}(z_2).\notag
\end{align*}
Similarly, we introduce the matrix
\begin{equation*}
\widetilde{\G}=\{\widetilde{\G}_{i,j}\}_{i,j=1}^n,\quad \widetilde{\G}_{i,j}=\sum_{s}H_{is,\,js}
\end{equation*}
and define functions
\begin{align}
&\widetilde{g}_n^{(1)}(z_1,z_2)=n^{-3}\Tr \widetilde{\G}(z_1)\widetilde{\G}(z_2),\quad
\widetilde{g}_n^{(2)}(z_1,z_2)= n^{-3}\sum_{i}\widetilde{\G}_{ii}(z_1)\widetilde{\G}_{ii}(z_2).\label{gtild}
\end{align}
It follows from (\ref{form2}) that
\begin{align}
n\mathbf{E}_\alpha\big\{((H(z) Y_\al,Y_\al)_{\alpha}^{\circ })^2\big\}=2a (g_n^\al(z))^2&+2(g_n^{(1)}(z,z)+\widetilde{g}_n^{(1)}(z,z))
\label{form22}
\\
&+b(g_n^{(2)}(z,{z})+\widetilde{g}_n^{(2)}(z,z))+O(n^{-1}).\notag
\end{align}
We have:
\begin{lemma}
\label{l:f1f2}
Under conditions of Theorem \ref{t:main}, we have for $i=1,2$:
\begin{align}
&\V\{g_n^{(i)}\},\,\V\{\widetilde{g}_n^{(i)}\}=O(n^{-2}),\label{Vargi}
\\
&\lim_{n\rightarrow\infty}\E\{g_n^{(i)}(z_1,z_2)\}=\lim_{n\rightarrow\infty}\E\{\widetilde{g}_n^{(i)}(z_1,z_2)\}=f(z_1)f(z_2),\label{limgi}
\end{align}
where $f$ is the solution of (\ref{MPE}).
\end{lemma}

\begin{proof}
We prove the lemma for $g_n^{(1)}$, the cases of $\widetilde{g}_n^{(2)}$, $g_n^{(2)}$, and $\widetilde{g}_n^{(2)}$ can be treated similarly.
Without loss of generality  we can assume that in the definitions of $\G$ and $g_n^{(1)}$,  $H=G$. It follows from (\ref{mart2}) that
\begin{align*}
\V\{g_n^{(1)}\}\le \sum_{\al} \E\{|(g_n^{(1)})^\circ_\al|^2\}.
\end{align*}
We have
\begin{align*}
g_n^{(1)}-g_n^{(1)\al}=&n^{-3}\Tr (\G(z_1)-\G^\al(z_1))\G(z_2)
\\
&+n^{-3}\Tr \G^\al(z_1)(\G(z_2)-\G^\al(z_2))=:S_n^{(1)}+S_n^{(2)}.
\end{align*}
Hence
\begin{align*}
(g_n^{(1)})^\circ_\al=g_n^{(1)}-g_n^{(1)\al}- \E_\al\{g_n^{(1)}-g_n^{(1)\al}\}=(S_n^{(1)})^\circ_\al+(S_n^{(2)})^\circ_\al,
\end{align*}
 and to get (\ref{Vargi}), it is enough to show that
\begin{align}
\label{Sni}
\E\{|S_n^{(j)}|^2\}=O(n^{-4}),\quad j=1,2.
\end{align}
Consider $S_n^{(1)}$.
It follows from (\ref{G-G}) that
\begin{align}
S_n^{(1)}= A_\al^{-1} n^{-3}\sum_{s,p}\sum_{j}\G_{s,p}(H^\al \Y_\al)_{js}(H^\al \Y_\al)_{jp}.
\label{Sn1}
\end{align}
Since for $x,\xi\in \R^n$ and an $n\times n$ matrix $D$
\begin{align}
\label{Hij}
\Big|\sum_{i,j}D_{ij}x_i\xi_j\Big|\le ||D||\cdot||x||\cdot||\xi||,
\end{align}
taking into account  $||H||\le 1/|\Im z|$, (\ref{A>}), and (\ref{normG}) we get
\begin{align}
|S_n^{(1)}|&\le n^{-3}(1+|\tau_\al|\cdot|\Im z|^{-1}||Y_\al||^2)\cdot||\G||\cdot||H^\al Y_\al||^2\notag
\\
&\le n^{-2}(1+|\tau_\al|\cdot|\Im z|^{-1}||Y_\al||^2)|\Im z|^{-3}|| Y_\al||^2.\label{Sn1<}
\end{align}
This and following from  (\ref{iso}) and (\ref{Ayy6}) bound
\begin{align}
\label{12}
\E\{||Y_\al||^p\}\le C,\quad p\le 12
\end{align}
 imply (\ref{Sni}) for $j=1$. The case $j=2$ can be treated similarly. So we get (\ref{Vargi}) for $g_n^{(1)}$.

\medskip

Let us prove (\ref{limgi}) for $g_n^{(1)}$. Let $f_n^{(1)}=\E\{g_n^{(1)}\}$. For a convergent subsequence $\{f_{n_i}^{(1)}\}$,
put $f^{(1)}:=\lim_{n_i\rightarrow\infty}f_{n_i}^{(1)}$. It follows from (\ref{G-G}) that
$$
(\Y_\al \Y_\al^TH)_{\jj,\,\q}=A_\al^{-1}\Y_{\al \jj}(H^\al \Y_\al)_{\q}.
$$
This and the resolvent identity yield
\begin{align*}
H_{\jj,\,\q}(z_1)=-z_1^{-1}\delta_{\jj,\,\q}+z_1^{-1}\sum_{\al}\tau_\al A_\al^{-1}(z_1)\Y_{\al \jj}(H^\al (z_1) \Y_\al)_{\q}.
\end{align*}
Hence,
\begin{align*}
z_1f_n^{(1)}(z_1,z_2)=&-f_n(z_2)+n^{-3}\sum_{\jj,\pp}\sum_{\al}\tau_\al \E\Big\{\frac{\Y_{\al \jj}(H^\al(z_1) \Y_\al)_{j_1p_2}}{A_\al(z_1)}H^\al_{\pp,\,p_1j_2} (z_2)\Big\}
\\
&-n^{-3}\sum_{\jj,\pp}\sum_{\al}\tau_\al^2 \E\Big\{\frac{\Y_{\al \jj}(H^\al(z_1) \Y_\al)_{j_1p_2}}{A_\al(z_1)}
\cdot\frac{(H^\al(z_2)\Y_{\al})_\pp(H^\al(z_2) \Y_\al)_{p_1j_2}}{A_\al(z_2)}\Big\}
\\
&=-f_n(z_2)+T_n^{(1)}+T_n^{(2)}.
\end{align*}
By the H$\ddot{\text{o}}$lder inequality, (\ref{A>}),  and (\ref{12})
\begin{align*}
|T_n^{(2)}|&\le n^{-3}\sum_{\al}\tau_\al^2
\E\Big\{\frac{||\Y_{\al }||\cdot||H^\al(z_1) \Y_\al||}{|A_\al(z_1)|}
\cdot\frac{||H^\al(z_2) \Y_\al||\cdot||H^\al(z_2) \Y_\al||}{|A_\al(z_2)|}\Big\}
\\
&\le Cn^{-3}\sum_{\al}\tau_\al^2
\E\{||\Y_{\al }||^4|A_\al(z_1)|^{-1}|A_\al(z_2)|^{-1}\}=O(n^{-1}).
\end{align*}
It follows from (\ref{iso}) that
$$
\E_\al\{\Y_{\al \jj}(H^\al \Y_\al)_{j_1p_2}\}=n^{-2}H^\al_{\jj,\,j_1p_2}.
$$
This and (\ref{1/A}) yield
\begin{align*}
T_n^{(1)}&=n^{-5}\sum_{\jj,\pp}\sum_{\al}\tau_\al \frac{\E\{H^\al_{\jj,\,j_1p_2}H^\al_{\pp,\,p_1j_2} (z_2)\}}{1+\tau_\al f_n^\al(z_1)}+r_n,
\\
r_n&=n^{-3}\sum_{\jj,\pp}\sum_{\al}\frac{\tau_\al}{\E\{A_\al(z_1)\}} \E\Big\{A^\circ_\al(z_1)
\frac{\Y_{\al \jj}(H^\al(z_1) \Y_\al)_{j_1p_2}}{A_\al(z_1)}H^\al_{\pp,\,p_1j_2} (z_2)\Big\}.
\end{align*}
Treating $r_n$ we note that
$$
n^{-1}\sum_{\jj,p_2}\big| \Y_{\al \jj}(H^\al \Y_\al)_{j_1p_2}\mathcal{G}^\al_{p_2,j_2}\big|\le
n^{-1}||\mathcal{G}^\al||\cdot ||Y_\al||\cdot ||H^\al Y_{\al}||\le C ||Y_\al||^2.
$$
Hence, by the Schwarz inequality, (\ref{A>}), (\ref{EA>}), (\ref{A0B0}), and (\ref{12})
\begin{align*}
 |r_n|&\le C n^{-2}\sum_{\al}\E\{|A^\circ_\al|\cdot|A_\al|^{-1}||Y_\al||^2\}
 \\
 &\le C n^{-2}\sum_{\al}\E\{|A^\circ_\al|^2\}^{1/2}\E\{|A_\al|^{-2}||Y_\al||^4\}^{1/2} =O(n^{-1/2}).
\end{align*}
    Also one can replace $f_n^\al$ and $H^\al$ with $f_n$ and $G$ (the error term is of the order $O(n^{-1})$). Hence,
\begin{align*}
z_1f_n^{(1)}(z_1,z_2)=&-f_n(z_2)+f_n^{(1)}(z_1,z_2)n^{-2}\sum_{\al}
\frac{\tau_\al}{1+\tau_\al f_n(z_1)}+o(1).
\end{align*}
This, (\ref{mn}), (\ref{sigma}), and (\ref{MPE}) leads to
$$
f^{(1)}(z_1,z_2)=f(z_2)\Big(c\int\frac{\tau d\sigma(\tau)}{1+\tau f(z_1)}-z_1\Big)^{-1}=f(z_1)f(z_2)
$$
 and  finishes the proof of the lemma.
\end{proof}

\medskip

It   follows from Lemmas \ref{l:form} and \ref{l:f1f2} that
under conditions of Theorem \ref{t:main}
\begin{align}
&\lim_{n\rightarrow\infty}n\tau_\al^{-2}\mathbf{E}\{A_{\al}^\circ(z_1)A_{\al}(z_2)\}=2(a+b+2)f(z_{1})f(z_2),
\label{limAA}
\end{align}
where $f$ is the solution of (\ref{MPE}).

\medskip

\begin{lemma}
\label{l:crocodile} Under conditions of Theorem \ref{t:main}
\begin{align}
\mathbf{Var}\{\mathbf{E}_\alpha\{(A_{\alpha}^\circ)^p\}\}=O(n^{-4}),\quad p=2,3.  \label{varEA}
\end{align}
\end{lemma}
\begin{proof}
 Since $\tau_\al$, $\al\in[m]$, are uniformly bounded in $\al$ and $n$, then to get the desired bounds it is enough to consider the case
  $\tau_\al=1$, $ \al\in[m]$. By (\ref{A0Aa0}), we have
\begin{align*}
&\E_\al\{(A_{\al}^\circ)^2\}=\E_\al\{(H Y_\al,Y_\al)_\al^{\circ 2}\}+(g_n^\al)^{\circ 2},
\\
&\E_\al\{(A_{\al}^\circ)^3\}=\E_\al\{(H Y_\al,Y_\al)_\al^{\circ 3}\}+3\E_\al\{(H Y_\al,Y_\al)_\al^{\circ 2}\}g_n^{\al\circ }+(g_n^\al)^{\circ 3},
\end{align*}
where by (\ref{g4<}) $\E\{|(g_n^\al)^{\circ }|^{2p}\}=O(n^{-6})$, $p=2,3$,  and by (\ref{EA0p<}) and (\ref{varg})
$$
\E\{|\E_\al\{(H Y_\al,Y_\al)_\al^{\circ 2}\}(g_n^\al)^{\circ }|^2\}=O(n^{-2})\E\{|(g_n^\al)^{\circ }|^2\}=O(n^{-5}).
$$
Hence,
$$
\V\{\E_\al\{(A_{\al}^\circ)^p\}\}\le 2\V\{\E_\al\{(H Y_\al,Y_\al)_\al^{\circ p}\}\}+O(n^{-4}),\quad p=2,3.
$$
 It also follows from (\ref{form22})
and Lemmas \ref{l:varg} and \ref{l:f1f2} that
\begin{align}
\V\{\E_\al\{(H Y_\al,Y_\al)_\al^{\circ 2}\}\}=O(n^{-4}),  \label{var2}
\end{align}
which leads to  (\ref{varEA}) for $p=2$. To get (\ref{varEA}) for $p=3$, it is enough to show that
\begin{align}
\mathbf{Var}\{\mathbf{E}_\alpha\{(H Y_\al,Y_\al)_{\alpha}^{\circ 3}\}\}=O(n^{-4}).  \label{varF}
\end{align}
 We have
\begin{align*}
\E_\al\{(H Y_\al,Y_\al)_{\alpha}^{\circ 3}\}=&\E_\al\{(H Y_\al,Y_\al)^3\}-\E_\al\{(H Y_\al,Y_\al)\}^3
\\
&-3\E_\al\{(H Y_\al,Y_\al)\}\cdot\E_\al\{(H Y_\al,Y_\al)_\al^{\circ 2}\}
\\
&=\E_\al\{(H Y_\al,Y_\al)^3\}-g_n^{\al 3}-3g_n^\al\cdot\E_\al\{(H Y_\al,Y_\al)_\al^{\circ 2}\}.
\end{align*}
It follows from  (\ref{varg}),  (\ref{var2}), and (\ref{varprod}) with $h_1=g_n^\al$, $h_2=n\E_\al\{(H Y_\al,Y_\al)_\al^{\circ 2}\}$ that
$$
\V\{g_n^\al \cdot n\E_\al\{(H Y_\al,Y_\al)_\al^{\circ 2}\}\}
\le C\big(\V\{g_n^\al \}+\V\{ n\E_\al\{(H Y_\al,Y_\al)_\al^{\circ 2}\}\}\big)=O(n^{-2}).
$$
Hence,
\begin{align*}
\V\{\E_\al\{(H Y_\al,Y_\al)_{\alpha}^{\circ 3}\}\}\le 2\V\Big\{\E_\al\{(H Y_\al,Y_\al)^3\}-g_n^{\al 3}\Big\}+O(n^{-4}),
\end{align*}
and to get (\ref{varF}) for $p=3$ it is enough to show that
\begin{align}
\V\Big\{\E_\al\{(H Y_\al,Y_\al)^3\}-g_n^{\al 3}\Big\}=O(n^{-4}).  \label{F3}
\end{align}
We have
\begin{align}
\E_\al\{(H Y_\al,Y_\al)^3\}=\sum_{\ii,\jj,\pp,\q,\s,\ttt}H_{\ii,\,\jj}H_{\pp,\,\q}H_{\s,\,\ttt}\Lambda(\ii,\jj,\pp,\q,\s,\ttt),\label{croco}
\end{align}
where
\begin{align*}
\Lambda(\ii,\jj,\pp,\q,\s,\ttt)=\prod_{k=1}^2\E_\al\{(y_{\al}^{(k)})_{i_k}(y_{\al}^{(k)})_{j_k}(y_{\al}^{(k)})_{p_k}(y_{\al}^{(k)})_{q_k}
(y_{\al}^{(k)})_{s_k}(y_{\al}^{(k)})_{t_k}\}
\end{align*}
and by (\ref{m6})
\begin{align}
\label{6}
\Lambda(\ii,\jj,\pp,\q,\s,\ttt)=O(n^{-6}).
\end{align}
Also, due to the unconditionality of the distribution,  $\Lambda$ contains only even moments. Thus in the index pairs $\ii,\jj,\pp,\q,\s,\ttt\in[n]^2$, every index (both on the first positions and on the second positions)
 is repeated an even number of times.
 Hence, there are at most $6$ independent indices: $\le 3$ on the first positions (call them $i,j,k$) and $\le 3$ on the second positions (call them $u,v,w$).
  For every fixed set of independent indices, consider maps $\Phi$ from this set to the sets of index pairs $\{\ii,\jj,\pp,\q,\s,\ttt\}$.
   We call such maps  {\it the index schemes}. Let $|\Phi|$
 be the cardinality of the corresponding set of independent indices. For example,
 $$
 \Phi\,:\, \{i,j\,;\,u,v,w\} \rightarrow  \{(i,u),(i,v);\,(i,w),(i,u);\,(j,w),(j,v)\},\quad |\Phi|=5,
 $$
 is an index scheme with $5$ independent indices ($i,j$ on the first positions and $u,v,w$ on the second positions). The inclusion-exclusion principle
 allows to split the expression (\ref{croco}) into the sums over fixed sets of independent indices
 of cardinalities from $2$ to $6$ with the fixed coefficients depending on $a_{2,2,2}$, $a_{2,4}$, and $a_{6}$ in front of every such sum. We have
 \begin{align}
\E_\al\{(H Y_\al,Y_\al)^3\}&=\sum_{\ell=2}^6 S_\ell,\quad
S_\ell=\sum_{\Phi:\,|\Phi|=\ell}\sum H_{\ii,\,\jj}H_{\pp,\,\q}H_{\s,\,\ttt} \Lambda'(\Phi),\label{croco2}
\end{align}
 where the last sum is taken  over the set of independent indices of cardinality $\ell$, $\Phi$ is an index scheme constructing  pairs
 $\{\ii,\jj,\pp,\q,\s,\ttt\}$ from this set,
 and $\Lambda'(\Phi)$ is a certain expression, depending on  $\Phi$, $a_{2,2,2}$, $a_{2,4}$, and $a_{6}$. For example,
 $$
 S_2=F(a_{2,2,2}, a_{2,4},a_{6})\sum_{i,u} (H_{iu,\,iu})^3,
 $$
 where $F(a_{2,2,2}, a_{2,4},a_{6})$ can be found by using the inclusion-exclusion formulas. As to $\Lambda'(\Phi)$ in (\ref{croco2}),
 the only thing we need to know is that
 \begin{align}
\label{l6}
\Lambda'(\Phi)=O(n^{-6}),
\end{align}
and that in the particular case of
$$
 \Phi_{\Tr}\,:\, \{i,j,k\,;\,u,v,w\} \rightarrow  \{(i,u),(i,u);\,(j,v),(j,v);\,(k,w),(k,w)\},
 $$
 we have by (\ref{m6})
  \begin{align*}
\Lambda'(\Phi)=a_{2,2,2}^2=n^{-6}+O(n^{-7}),
\end{align*}
and the corresponding term in $S_6$ has the form $a_{2,2,2}^2(\Tr H)^3$.

 Note that by (\ref{6}), $S_2$ is of the order $O(n^{-4})$. By the same reason
 $$
 \Big|\sum_{\ell=2}^4S_\ell\Big |=O(n^{-2})
 $$
 so that
 $$
 \V\Big\{\Big|\sum_{\ell=2}^4S_\ell\Big|\Big\}=O(n^{-4}).
 $$
Hence to get (\ref{F3}) it suffices to consider terms with $5$ and $6$ independent indices and show that
\begin{align}
\V\{S_5\},\,\,\V\big\{S_6-g_n^{\al 3}\big\}=O(n^{-4}).  \label{S56}
\end{align}

 Consider $S_5$.  In this case we have exactly $5$ independent indices. By the symmetry we can suppose that  there are
two  first independent indices, $i,j$,  and three second independent indices, $u, v, w$, and that we have $i$ on four places and $j$ on two places. Thus, $S_5$
is equal to the sum of terms of the form
\begin{align*}
&S_5'=O(n^{-6})\sum_{i,j,u,v,w}H_{i\cdot,\,i\cdot}H_{i\cdot,\,j\cdot}H_{i\cdot,\,j\cdot}\quad\text{or}\quad
\\
&S_5''=O(n^{-6})\sum_{i,j,u,v,w}H_{i\cdot,\,i\cdot}H_{i\cdot,\,i\cdot}H_{j\cdot,\,j\cdot}\,.
\end{align*}
Here we suppose that there are some fixed indices on the dot-places, which are different from explicitly mentioned ones.
Note that  $S_5'$ has a single "external" pairing with respect to $j$.
While estimating the terms, our argument is essentially based on the simple
relations
\begin{align}
\sum_{j,v}|H_{iu,\,jv}|^2=O(1),\quad  |H_{iu,\,jv}|=O(1), \quad  ||H||=O(1),\label{main}
\end{align}
and on the observation that the more mixing of matrix entries we have the lower order of sums we get.
Let $V\subset \R^n$ be the set of vectors of the form
$$
\xi=\{\xi_j\}_{j=1}^n=\{H_{\cdot\cdot,\,j\cdot}\}_{j=1}^n\quad\text{or}\quad \xi=\{H_{\cdot\cdot,\,\cdot u}\}_{u=1}^n,
$$
and let $W$ be the set of $n\times n$ matrices of the form
$$
D=\{H_{i\cdot,\,j\cdot}\}_{i,j=1}^n,\quad\text{or}\quad D=\{H_{i\cdot,\,\cdot u}\}_{i,u=1}^n,\quad\text{or}\quad
 D=\{H_{\cdot u,\,\cdot v}\}_{u,v=1}^n.
$$
It follows from (\ref{main}) that
$$
\forall\xi\in V\quad ||\xi||=O(1)\quad\text{and}\quad\forall D\in W\quad ||D||=O(1).
$$
Hence,
\begin{align}
&\sum_{j}|H_{\cdot\cdot,\,j\cdot}H_{\cdot\cdot,\,j\cdot}|=O(1),\quad \sum_{u}|H_{\cdot\cdot,\,\cdot u}H_{\cdot\cdot,\,\cdot u}|=O(1), \label{dots}
\\
&\sum_{i,j}H_{i\cdot,\,j\cdot}H_{i\cdot,\,\cdot\cdot}H_{\cdot\cdot,\,j\cdot}=O(1),\quad\text{and}\quad
\sum_{i,u}H_{i\cdot,\,\cdot u}H_{i\cdot,\,\cdot\cdot}H_{\cdot\cdot,\,\cdot u}=O(1).\label{form}
\end{align}
In particular, by (\ref{main}) and (\ref{dots}), we have for $S_5'$
$$
|S_5'|\le O(n^{-6})\sum_{i,u,v,w}\sum_{j}|H_{i\cdot,\,j\cdot}H_{i\cdot,\,j\cdot}|=O(n^{-2}),
$$
so that $\V\{S_5'\}=O(n^{-4})$. Consider $S_5''$. Note that if in $S_5''$ we have a single "external" pairing with respect to at least one index
on the second positions, then similar to $S_5'$, the variance of this term is of the order $O(n^{-4})$. So we are left with the terms of the form
 \begin{align*}
S_5'''=O(n^{-6})\sum_{i,j,u,v,w}H_{iu,\,iu}H_{iv,\,iv}H_{jw,\,jw}.
\end{align*}
It follows from (\ref{gtild}) that
\begin{align*}
S_5'''=O(n^{-1})\cdot g_n^\al(z)\cdot\widetilde{g}_{n}^{(2)}(z,z).
\end{align*}
Now (\ref{varprod}), (\ref{varg}), and (\ref{Vargi}) imply that
\begin{align*}
\V\{S_5'''\}\le C n^{-2}(\V\{g_n^\al\}+\V\{\widetilde{g}_{n}^{(2)}\})=O(n^{-4}).
\end{align*}
Summarizing we get $\V\{S_5\}=O(n^{-4})$.

\medskip

Consider $S_6$ and show that $\V\{S_6-g_n^{\al 3}\}=O(n^{-4})$. In this case we have $6$ independent indices, $i,j,k$ for the first positions
and $u,v,w$ for the second positions.  Suppose that we have two single external pairing with respect to two different first indices and consider terms of the form
\begin{align*}
&S_6'=O(n^{-6})\sum_{i,j,k,u,v,w}H_{i\cdot,\,j\cdot}H_{i\cdot,\,\cdot\cdot}H_{\cdot\cdot,\,j\cdot},
\\
&S_6''=O(n^{-6})\sum_{i,j,k,u,v,w}H_{i\cdot,\,j\cdot}H_{i\cdot,\,j\cdot}H_{\cdot\cdot,\,\cdot\cdot}\,.
\end{align*}
It follows from (\ref{form}) that $S_6'=O(n^{-2})$, hence
$
\V\{S_6'\}=O(n^{-4}).
$
Consider $S_6''$
\begin{align}
&S_6''=O(n^{-6})\sum_{i,j,k,u,v,w}H_{i\cdot,\,j\cdot}H_{i\cdot,\,j\cdot}H_{k\cdot,\,k\cdot}\,.\label{S6}
\end{align}
If the second indices in $H_{k\cdot,\,k\cdot}$ are not equal, then we get the expression of the form
\begin{align*}
&S_6'''=O(n^{-6})\sum_{i,j,k,u,v,w}H_{i\cdot,\,ju}H_{i\cdot,\,j\cdot}H_{k\cdot,\,ku}.
\end{align*}
 It follows from (\ref{form}) that $S_6'''=O(n^{-2})$, hence
$
\V\{S_6'''\}=O(n^{-4}).
$
 If the second indices in $H_{k\cdot,\,k\cdot}$ in (\ref{S6}) are equal, then we get
 the expressions of three types:
\begin{align*}
&O(n^{-6})\sum_{i,j,k,u,v,w}H_{iu,\,jv}H_{iu,\,jv}H_{kw,\,kw}=g_n^\al n^{-4}\sum_{i,j,u,v}(H_{iu,\,jv})^2=O(n^{-2}),
\\
&O(n^{-6})\sum_{i,j,k,u,v,w}H_{iu,\,jv}H_{iv,\,ju}H_{kw,\,kw}=g_n^\al n^{-4}\sum_{i,j,u,v}H_{iu,\,jv}H_{iv,\,ju}=O(n^{-2}),
\\
&O(n^{-6})\sum_{i,j,k,u,v,w}H_{iu,\,ju}H_{iv,\,jv}H_{kw,\,kw}=O(n^{-1})g_n^{(1)}(z,z)g_n^\al(z),
\end{align*}
where we used (\ref{main}) to estimate the first two expressions, so that their variances are of the order $O(n^{-4})$.
It also follows from  (\ref{varprod}), (\ref{varg}), and (\ref{Vargi}) that the variance of the third expression
is  of the order $O(n^{-4})$. Hence, $\V\{S_6'''\}=O(n^{-4})$. It  remains to consider the term without external pairing, which corresponds to
$$
(a_{2,2,2})^2\sum_{i,j,k,u,v,w}H_{iu,\,iu}H_{jv,\,jv}H_{kw,\,kw}=(a_{2,2,2})^2\gamma_n^3
$$
(see (\ref{l6})). Summarizing we get
\begin{align*}
\V\{S_6-g_n^{\al 3}\}&\le2\V\{((a_{2,2,2})^2-n^{-6})\gamma_n^{\al 3}\}+O(n^{-4})
\\
&=O(n^{-2})\V\{g_n^{\al 3}\}+O(n^{-4})=O(n^{-4}),
\end{align*}
where we used (\ref{m6}) and (\ref{varg}). This leads to (\ref{S56}) and completes the proof of the lemma.
\end{proof}

\section{Covariance of the resolvent traces}
\label{s:covariance}

\begin{lemma}
\label{l:cov}
Suppose that the conditions of Theorem \ref{t:main} are fulfilled. Let
$$
C_n(z_1,z_2):=n^{-1}\mathbf{Cov}\{\gamma_n(z_1),\,\gamma_n(z_2)\}=n^{-1}\E\{\gamma_n(z_1)\gamma_n^\circ(z_2)\}.
$$
Then $\{C_n(z_1,z_2)\}_n$ converges uniformly in $z_{1,2}\in K$ to
\begin{align}
\label{limCov}
C(z_1,z_2)&=2(a+b+2)c
\int\frac{f'(z_1)  }{(1+\tau f(z_1))^2}\frac{f'(z_2)  }{(1+\tau f(z_2))^2}\tau^2d\sigma(\tau).
\end{align}
\end{lemma}

\begin{proof}
For a convergent subsequence $\{C_{n_i}\}$, denote
$$
C(z_1,z_2):=\lim_{n_i\rightarrow\infty}C_{n_i}(z_1,z_2).
$$
We will show that for every converging subsequence, its limit satisfies (\ref{limCov}).
Applying the resolvent identity, we get  (see (\ref{gn=}))
\begin{align}
C_n(z_1,z_2)=&-\frac{1}{nz_1}\sum_\al\E\{A_\al^{-1}(z_1)\gamma_n^\circ(z_2)\}=
-\frac{1}{nz_1}\sum_\al\E\{A_\al^{-1}(z_1)\gamma_n^{\al\circ}(z_2)\}\notag
\\
&
-\frac{1}{nz_1}\sum_\al\E\{A_\al^{-1}(z_1)(\gamma_n-\gamma_n^{\al})^\circ(z_2)\}=:T_{n}^{(1)}+T_{n}^{(2)}.\label{Cn=}
\end{align}
Consider $T_{n}^{(1)}$. Iterating (\ref{1/A}) four times, we get
\begin{align*}
 T_n^{(1)}=\frac{1}{nz_1}\sum_\al\Big[&\frac{\E\{{A_\al}(z_1)\gamma_n^{\al\circ}(z_2)\}}{\E\{A_\al(z_1)\}^2}
 -\frac{\E\{{A_\al^{\circ 2}}(z_1)\gamma_n^{\al\circ}(z_2)\}}{\E\{A_\al(z_1)\}^3}
 +\frac{\E\{{A_\al^{\circ 3}}(z_1)\gamma_n^{\al\circ}(z_2)\}}{\E\{A_\al(z_1)\}^4}
 \\
 &-\frac{\E\{A_\al^{-1}(z_1){A_\al^{\circ 4}}(z_1)\gamma_n^{\al\circ}(z_2)\}}{\E\{A_\al(z_1)\}^4}\Big]=:
 S_n^{(1)}+S_n^{(2)}+S_n^{(3)}+S_n^{(4)}.
\end{align*}
It follows from (\ref{EA>}), (\ref{varg}), and (\ref{varEA}) that $S_n^{(i)}=O(n^{-1/2})$, $i=2,3$. Also, by (\ref{A>})  we have
$$
{\E\{|A_\al^{-1}A_\al^{\circ 4}\gamma_n^{\al\circ}|\}}\le {\E\{(1+|\tau_\al|||Y_\al||^2/|\Im z|)|A_\al^{\circ 4}\gamma_n^{\al\circ}|\}},
$$
 where by the Schwarz
inequality, (\ref{g4<}), (\ref{EA0p<}), and (\ref{12})
$$
{\E\{||Y_\al||^2|A_\al^{\circ 4}\gamma_n^{\al\circ}|\}}\le
\E\{|A_\al^{\circ }|^{6}\}^{1/2}\E\{|A_\al^{\circ }|^{4}\}^{1/4}\E\{\E_\al\{||Y_\al||^8\}|\gamma_n^{\al\circ}|^4\}^{1/4}=O(n^{-3/2}).
$$
Hence $S_n^{(4)}=O(n^{-1/2})$, and we are left with $S_n^{(1)}$. We have
 \begin{align*}
\E\{{A_\al}(z_1)\gamma_n^{\al\circ}(z_2)\}&=\E\{\E_\al\{{A_\al}(z_1)\}\gamma_n^{\al\circ}(z_2)\}
=n^{-2}\tau_\al\E\{\gamma_n^{\al\circ}(z_1)\gamma_n^{\al\circ}(z_2)\}.
 \end{align*}
 It follows from (\ref{AB}) and (\ref{B/A<}) that $|\gamma_n(z)-\gamma^\al_n(z)|\le 1/|\Im z|$. This and (\ref{varg}) yield
 \begin{align*}
|\E\{(\gamma^\al_n-\gamma_n)^\circ(z_1)\gamma_n^{\al\circ}(z_2)\}&+\E\{(\gamma_n^{\circ}(z_1)(\gamma^\al_n-\gamma_n)^\circ(z_2)\}|
\\
&\le \E\{|\gamma_n^{\al\circ}(z_2)|\}/|z_1|+\E\{|\gamma_n^{\circ}(z_1)|\}/|z_2|=O(n^{1/2}).
 \end{align*}
 Hence,
 \begin{align*}
\E\{{A_\al}(z_1)\gamma_n^{\al\circ}(z_2)\}=  n^{-1}\tau_\al   C_n(z_1,z_2)+O(n^{-3/2}),
 \end{align*}
and we have
 \begin{align*}
S_n^{(1)}=C_n(z_1,z_2)\frac{1}{n^2z_1}\sum_\al\frac{\tau_\al}{(1+\tau_\al f_n^\al(z_1))^2}+O(n^{-1/2}).
\end{align*}
Summarizing, we get
\begin{align}
 T_n^{(1)}=C(z_1,z_2)\frac{c}{z_1}\int\frac{\tau d\sigma(\tau)}{(1+\tau f(z_1))^2}+o(1)\label{T1}.
\end{align}
Consider now $T_n^{(2)}$ of (\ref{Cn=}). By (\ref{AB}),
\begin{align}
T_{n}^{(2)}=\frac{1}{nz_1}\sum_\al\E\{A^{-1}_\al(z_1)(B_\al/A_\al)^\circ(z_2)\}.\label{Tn2}
\end{align}
For shortness let for the moment $A_i=A_\al(z_i)$, $i=1,2$, $B_2=B_\al(z_2)$. Iterating (\ref{1/A}) with respect to $A_1$ and $A_2$  two times we get
\begin{align*}
\E\{(1/A_1)^\circ(B_2/A_2)^\circ\}=&\frac{\E\{(-A_1^\circ+A_1^{-1}A_1^{\circ2})(B_2\E\{A_2\}-B_2A_2^\circ+B_2A_2^{-1}A_2^{\circ2})^\circ\}}
{\E\{A_1\}^2\E\{A_2\}^2}
\\
=&\frac{-\E\{A_1^\circ B_2\}\E\{A_2\}+\E\{B_2\}\E\{A_1^\circ A_2\}}
{\E\{A_1\}^2\E\{A_2\}^2}
\\
&\hspace{-3cm}+
\frac{\E\{A_1^\circ B^\circ_2A_2^\circ-A_1^\circ B_2A_2^{-1}A_2^{\circ2}+A_1^{-1}A_1^{\circ2}(B_2\E\{A_2\}-B_2A_2^\circ+B_2A_2^{-1}A_2^{\circ2})^\circ\}}
{\E\{A_1\}^2\E\{A_2\}^2}.
\end{align*}
Applying (\ref{Ayy6}), (\ref{12}), and using bounds
 (\ref{B/A<}), (\ref{A>}), (\ref{EA>}) for $|B_2/A_2|$, $|A_i|^{-1}$, $|\E\{A_i\}|^{-1}$, $i=1,2$, one can show that the terms containing at least three centered factors $A_1^\circ$, $A_2^\circ$, $B_2^\circ$ are of the order $O(n^{-3/2})$. This implies that
\begin{align*}
\E\{(1/A_1)^\circ(B_2/A_2)^\circ\}=\frac{-\E\{A_1^\circ B_2\}\E\{A_2\}+\E\{B_2\}\E\{A_1^\circ A_2\}}
{\E\{A_1\}^2\E\{A_2\}^2}+O(n^{-3/2}).
\end{align*}
Returning to the original notations and taking into account that
$$
B_\al(z)=\partial A_\al(z)/\partial z,
$$
 we get
\begin{align}
\label{Tn22}
\E\{A^{-1}_\al(z_1)(B_\al/A_\al)^\circ(z_2)\}=-\frac{1} {\E\{A_{\alpha }(z_1) \}^2}\frac{\partial}{%
\partial z_2} \frac{\E\{A_{\alpha }^\circ(z_1)A_{\alpha
}^\circ(z_2)\}} {\E\{A_{\alpha }(z_2) \}}+O(n^{-3/2}).
\end{align}
Denote for the moment
$$
D=2(a+b+2).
$$
It follows from (\ref{limAA}) and (\ref{Tn2}) -- (\ref{Tn22}) that
\begin{align*}
 T_n^{(2)}=-\frac{Dc}{z_1}\int\frac{\tau^2 f(z_1) }{(1+\tau f(z_1))^2}\frac{\partial}{%
\partial z_2}\frac{f(z_2)}{1+\tau f(z_2)}d\sigma(\tau)+o(1).
\end{align*}
This and (\ref{Cn=}) -- (\ref{T1}) yield
\begin{align*}
C(z_1,z_2)=\frac{Dc}{
c\int{\tau } {(1+\tau f(z_1))^{-2}}d\sigma (\tau)-z_1}%
\int\frac{\tau^2 f(z_1) }{(1+\tau f(z_1))^2}\frac{\partial}{%
\partial z_2}\frac{f(z_2)}{1+\tau f(z_2)}d\sigma(\tau).
\end{align*}
Note that by (\ref{MPE}),
$$
c\int\frac{\tau d\sigma (\tau) } {(1+\tau f(z))^2}-z=\frac{f(z)}{f'(z)}.
$$
Hence
\begin{align*}
C(z_1,z_2)&=Dc
\int\frac{f'(z_1)  }{(1+\tau f(z_1))^2}\frac{f'(z_2)  }{(1+\tau f(z_2))^2}\tau^2d\sigma(\tau).
\end{align*}
which completes the proof of the lemma.
\end{proof}

\section{Proof of Theorem \ref{t:main}}
\label{s:proof}
The proof essentially repeats the proofs of Theorem $1$ of \cite{Sh:11} and Theorem 1.8 of \cite{GLPP:13}, the technical details are provided by the calculations of the proof of Lemma \ref{l:cov}. It suffices to show that if
\begin{equation}
\label{Z_n}
Z_n(x)=\mathbf{E}\{e_{n}(x)\},\quad e_{n}(x)=e^{ix\mathcal{N}_n^\circ[\varphi]/\sqrt{n}},
\end{equation}%
then we have uniformly in $|x|\le C$
\begin{equation*}  
\lim_{n\rightarrow\infty}Z_n(x)=\exp\{-x^2V[\varphi]/2\}
\end{equation*}
with $V[\varphi]$ of (\ref{Var}).
Define for every test functions $\varphi\in \HH_{s}$, $s>5/2$,
\begin{equation}  \label{phi_y}
\varphi_\eta=P_\eta*\varphi,
\end{equation}
where $P_\eta$ is the Poisson kernel
\begin{equation}  \label{P_y}
P_\eta(x)=\frac{\eta}{\pi(x^2+\eta^2)},
\end{equation}
and "$*$" denotes the convolution. We have
\begin{equation}  \label{appr}
\lim_{\eta\downarrow 0}||\varphi-\varphi_\eta||_{s}=0.
\end{equation}
Denote for the moment the characteristic function (\ref{Z_n}) by $%
Z_n[\varphi]$, to make explicit its dependence on the test function. Take any
converging subsequence $\{Z_{n_j}[\varphi]\}_{j=1}^\infty$
Without loss of generality assume that  the whole sequence $\{Z_{n_j}[\varphi_{\eta}]\}$ converges as $n_j\rightarrow\infty$. By (\ref{apriory}), we have
\begin{equation*}
|Z_{n_j}[\varphi]-Z_{n_j}[\varphi_{\eta}]|\le |x|n^{-1/2}\big(\mathbf{Var}\{\mathcal{N}%
_{n_j}[\varphi] -\mathcal{N}_{n_j}[\varphi_\eta]\}\big)^{1/2}\le
C|x|||\varphi-\varphi_\eta||_{s},
\end{equation*}
hence
$$
\lim_{\eta\downarrow 0}\lim_{{n_j}\rightarrow\infty}(Z_{n_j}[\varphi]-Z_{n_j}[\varphi_{\eta}])=0.
$$
This and the equality $Z_{n_j}[\varphi]=(Z_{n_j}[\varphi]-Z_{n_j}[\varphi_{\eta}])+Z_{n_j}[\varphi_{\eta}]$ imply that
\begin{equation}
\label{limny}
\exists\lim_{\eta\downarrow 0}\lim_{{n_j}\rightarrow\infty}Z_{n_j}[\varphi_{\eta}]
\quad\text{and}\quad
\lim_{{n_j}\rightarrow\infty}Z_{n_j}[\varphi]=\lim_{\eta\downarrow 0}\lim_{{n_j}%
\rightarrow\infty}Z_{n_j}[\varphi_{\eta}].
\end{equation}
Thus it suffices to find the limit of
$$
Z_{\eta n}(x):=Z_{n}[\varphi_{\eta}]=\mathbf{%
E}\{e_{\eta n}(x)\}, \quad\mbox{where}\quad
e_{\eta n}(x)=e^{ix\mathcal{N}_n^\circ[\varphi_\eta ]/\sqrt{n}},
$$
 as ${n} \rightarrow\infty$.
It follows from (\ref{phi_y}) -- (\ref{P_y}) that
\begin{equation}  \label{repr_N}
\mathcal{N}_n[\varphi_ \eta ]=\frac{1}{\pi}\int
\varphi(\mu)\Im\gamma_n(z)d\mu,\quad z=\mu+i\eta .
\end{equation}
This allows to write
\begin{equation}
\frac{d}{dx}Z_{\eta n}(x)=\frac{1}{2\pi}\int \varphi(\mu)(
\YY(z,x)- \YY(\overline z,x))d\mu,  \label{dZ=}
\end{equation}
where
$$
\YY(z,x)=n^{-1/2}\mathbf{E}\{\gamma_n(z)e_{\eta n}^\circ(x)\}.
$$ 
Since $|\YY(z,x)|\le2n^{-1/2}\mathbf{Var}\{\gamma_n(z)\}^{1/2}$, it follows from the proof of Lemma \ref{l:apriory} that  for every $\eta>0$ the integrals of $|\YY(z,x)|$ over $\mu$ are uniformly bounded in $n$. This and the fact that $\varphi\in L^2$ together with Lemma \ref{l:Vitali} below show that to find the limit of integrals in (\ref{dZ=})  it is enough to find the pointwise limit of $\YY(\mu+i\eta,x)$.
 We have
\begin{align*}
\YY(z,x)&=-\frac{1}{zn^{1/2}}\sum_{\alpha=1}^m\big[\mathbf{E}\{A^{-1}_{\alpha}(z)e_{\eta n}^{\al\circ}(x)\} -
\mathbf{E}\{A^{-1}_{\alpha}(z)(e^\circ_{\eta n}(x)-e_{\eta n}^{\al\circ}(x))\}\big],
\end{align*}
where $e_{\eta n}^{\al}(x)=\exp\{ix\mathcal{N}_n^{\alpha\circ}[\varphi_\eta ]/\sqrt{n}\}$ and
$\mathcal{N}_n^{\alpha}[\varphi_\eta ]=\Tr \varphi_\eta (M^\alpha)$.
By (\ref{repr_N}),
\begin{align*}
e_{\eta n}-e_{\eta n}^\alpha=&\frac{ixe_{\eta n}^\alpha}{\sqrt{n}\pi}\int _{  %
}\varphi(\lambda_1)\Im ( \gamma
_{n}-\gamma_{n}^{\alpha })^\circ(z_1)d\lambda_1
\\
&+O\Big(\Big|\frac{1}{\sqrt{n}}\int \varphi(\lambda_1)\Im ( \gamma
_{n}-\gamma_{n}^{\alpha })^\circ(z_1)d\lambda_1\Big|^2\Big),
\end{align*}
so that
\begin{align*}
\mathbf{E}\{A_{\alpha n}^{-1}(z)(e_{\eta n}-e_{\eta n}^\alpha)^\circ(x)\}=&\frac{ixe_{\eta n}^\alpha}{\sqrt{n}\pi}\int _{  %
}\varphi(\lambda_1)\Im ( \gamma
_{n}-\gamma_{n}^{\alpha })^\circ(z_1)d\lambda_1  \notag \\
&+\int \int
O(R_n)\varphi(\lambda_1)%
\varphi(\lambda_2)d\lambda_1 d\lambda_2,
\end{align*}
where $z_j=\lambda_j+i\eta$, $j=1,2$, and
\begin{align*}
R_n= n^{-1}\mathbf{E}\{(A_{\alpha n}^{-1})^\circ(z) \Im (B_{\alpha
n}A_{\alpha n}^{-1})^{\circ}(z_1) \Im (B_{\alpha n}A_{\alpha
n}^{-1})^{\circ}(z_2)\}.
\end{align*}
Using the argument of the proof of the Lemma \ref{l:cov}, it can be shown that $R_n= O(n^{-5/2})$. Hence,
\begin{align*}
\YY(z,x)=&-\frac{1}{zn^{1/2}}\sum_{\alpha=1}^m\mathbf{E}\{A^{-1}_{\alpha}(z)e_{\eta n}^{\al\circ}(x)\}
\\
&-\frac{ix}{zn\pi}\int \varphi(\lambda_1)\sum_{\alpha=1}^m
\mathbf{E}\{e_{\eta n}^\alpha(x) (A^{-1}_{\alpha}(z))^{\circ}
\Im ( \gamma_{n}-\gamma_{n}^{\alpha })^\circ(z_1)\}d\lambda_1+O(n^{-1}).
\end{align*}
Treating the r.h.s.  similarly to $T_n^{(1)}$ and $T_n^{(2)}$  of (\ref{Cn=}), we get
\begin{align}
\YY(z,x)=\frac{xZ_{\eta n}(x)}{2\pi }\int   \varphi(%
\lambda_1) \;[C(z,z_1)-C(z,\overline{z_1})]d\lambda_1+o(1),  \label{Yn=}
\end{align}
where $C(z,z_1)$ is defined in (\ref{limCov}). It follows from (\ref%
{dZ=}) and (\ref{Yn=}) that
\begin{align}
\frac{d}{dx}Z_{\eta n}(x)={-xV_{\eta }[\varphi]Z_{\eta n}(x)}+o(1),  \label{dZ}
\end{align}
(see (\ref{Var})) and finally
\begin{align*}
\lim_{n\rightarrow\infty}Z_{\eta n}(x)=\exp\{{-x^2V_{\eta }[\varphi]}/2\}.
\end{align*}
Taking into account (\ref{limny}), we pass to the limit
$\eta \downarrow 0$ and complete the proof of the theorem.
\qed

\medskip

It remains to prove the following lemma.
\begin{lemma}
\label{l:Vitali}
Let $g\in L^2(\R)$ and let $\{h_n\}\subset  L^2(\R)$ be a sequence of complex-valued functions such that
\begin{align*}
\int |h_n|^2dx<C\quad \text{and}\quad h_n\rightarrow h \quad a.e.\quad \text{as}\quad {n\rightarrow\infty},
\quad \text{where}\quad|h(x)|\le \infty\quad a.e.
\end{align*}
Then $$\int g(x)h_n(x)dx\rightarrow\int g(x)h(x)dx\quad \text{as}\quad n\rightarrow\infty.$$
\end{lemma}

\begin{proof}
According to the  convergence theorem of Vitali (see e.g. \cite{Rudin}), if $(X,\mathcal{F},\mu)$ is a positive measure space and
\begin{align*}
&\mu(X)<\infty,
\\
&\{F_n\}_n\quad \text{is uniformly integrable},
\\
&F_n\rightarrow F \quad a.e.\quad \text{as}\quad {n\rightarrow\infty},
\quad|F(x)|\le \infty\quad a.e.,
\end{align*}
then $F\in L^1(\mu)$ and $\lim_{n\rightarrow\infty}\int_X |F_n-F|d\mu=0$. Without loss of generality  assume that $g(x)\neq 0$, $x\in \R$, and take
$$
d\mu(x)=|g(x)|^2dx,\quad F_n=gh_n/|g|^2, \quad F=gh/|g|^2.
$$
Then
\begin{align*}
&\mu(\R)=\int |g(x)|^2dx<\infty,
\\
&\int_E|F_n(x)|d\mu(x)
\le||h_n||_{L^2}\Big(\int|g(x)|^2dx\Big)^{1/2}\le C(\mu(E))^{1/2},
\\
&F_n\rightarrow F \quad a.e.\quad \text{as}\quad {n\rightarrow\infty},
\quad|F(x)|\le \infty\quad a.e.
\end{align*}
Hence, the conditions of the Vitali's theorem are fulfilled and we get
$$
\lim_{n\rightarrow\infty}\int |F_n-F|d\mu=\lim_{n\rightarrow\infty}\int |h_n-h||g|dx=0,
$$
which completes the proof of the lemma.
\end{proof}

\medskip

{\bf Acknowledgements.} The author would like to thank Leonid Pastur for introducing to the problem and for the fruitful discussions.


\begin{thebibliography}{99}

\bibitem{Adam:11} Adamczak, R. (2011). \emph{On the Marchenko-Pastur and circular laws for some classes of
random matrices with dependent entries}. Electronic Journal of Probability, 16,  1065--1095.

\bibitem{Ak-Gl:93}  Akhiezer, N. I.,  Glazman,  I. M. \emph{Theory of Linear
Operators in Hilbert Space}, Dover, New York, 1993.

\bibitem{AHH:12} Ambainis, A., Harrow, A. W., and Hastings, M. B. (2012).
\textit{Random tensor theory: extending random matrix theory to random product states.} Commun. Math. Phys., \textbf{310}(1), 25--74.

\bibitem{Ba-Si:04} {Bai, Z. D.}, {Silverstein, J. W.} (2004). \textit{CLT for linear
spectral statistics of large dimensional sample covariance
matrices.} {Ann. Prob.}, \textbf{32},  553--605.

\bibitem{Ba-Si:10}   {Bai, Z. D.}, {Silverstein, J. W.} (2010). \textit{Spectral Analysis
of Large Dimensional Random Matrices}, Springer, New York.

\bibitem{Ba-Zh:08} Bai, Z. D., Zhou, W. (2008). \textit{Large sample covariance matrices without independence structures in columns}. Statistica Sinica, {\bf 18}(2), 425.

\bibitem{Ba-W-Zh:10}    Bai, Z. D., Wang, X., and Zhou, W. (2010). \textit{Functional CLT for sample covariance matrices}. Bernoulli, \textbf{16}(4), 1086--1113.

\bibitem{Ba-Me:13} Banna, M., Merlev{\'e}de, F. (2015). \textit{Limiting spectral distribution of large sample
covariance matrices associated to a class of stationary processes.} J. Theor. Prob. {\bf 28}(2) 745--783 

\bibitem{Cab:01} Cabanal-Duvillard, T. (2001). \textit{Fluctuations de la loi empirique de grandes matrices al´eatoires}. Ann. Inst. H.
Poincar´e Probab. Statist., \textbf{37}(3), 73--402.

\bibitem{Dh-Co:68}  Dharmadhikari, S. W.,  Fabian, V., and  Jogdeo, K. (1968). \emph{Bounds on the
moments of martingales}, {Ann. Math. Statist.} \textbf{39}, 1719--1723.

    \bibitem{Gi:01} {Girko, V.} (2001). \textit{Theory of Stochastic Canonical Equations},
vols.I, II Kluwer, Dordrecht.

 \bibitem{G-N-T:01} G$\ddot{\text{o}}$otze, F., Naumov, A.A., and Tikhomirov, A.N. (2014). \emph{Limit theorems for two classes
of random matrices with dependent entries}, Teor. Veroyatnost. i Primenen.,  \textbf{59}(1),  61--80.

\bibitem{GLPP:13} Gu{\'e}don, O., Lytova, A., Pajor, A., and Pastur, L.
(2014). {\it The Central Limit Theorem for linear eigenvalue statistics of the sum
of rank one projections on independent vectors}. Spectral Theory and Differential Equations. AMS Trans. Ser {\bf 2}(233), 145 -- 164. arXiv:1310.2506

\bibitem{Ha:09} Hastings., M. B. (2009). \textit{A counterexample to additivity of minimum output entropy}, Nature Physics, 5. arXiv:0809.3972.

\bibitem{Ha:07} Hastings., M. B. (2007). \textit{Entropy and entanglement in quantum ground states}, Phys. Rev. B, 76:035114. arXiv:cond-mat/0701055.

\bibitem{Ly-Pa:08}  Lytova, A.,   Pastur, L. (2009). \emph{Central limit theorem for linear
eigenvalue statistics of random matrices with independent entries}, {%
Ann.  Prob.} \textbf{37}(5),  1778--1840.



    \bibitem{Ma-Pa:67} {Marchenko, V.,} {Pastur, L.} (1967).  \textit{The eigenvalue
distribution in some ensembles of random matrices}. {Math.
USSR Sbornik}, \textbf{1}, 457--483.

\bibitem{Me-Pe:16} Merlev{\'e}de, F., Peligrad, M. (2016).  \textit{On the empirical spectral distribution for matrices
with long memory and independent rows}. Stochastic Processes and their Applications, arXiv:1406.1216

 \bibitem{Na-Yao:13} Najim, J., Yao, J. (2016).  \textit{Gaussian fluctuations for linear spectral statistics of large random covariance matrices.}  Ann. Appl. Prob. {\bf 26}(3), 1837-1887. arXiv:1309.3728.

\bibitem{Pa-Pa:09} Pajor, A. and Pastur, L. (2009). \textit{On the limiting empirical measure of eigenvalues of the sum of rank one matrices with log-concave distribution}, Studia Math., \textbf{195}(1), 11--29.

\bibitem{P-Zh:08} Pan, G. M., Zhou, W. (2008). \textit{Central limit theorem for signal-to-interference
ratio of reduced rank linear receiver}, {Ann. Appl. Probab.}, {\bf 18}, 1232--1270.

\bibitem{Pa:06b}  Pastur, L. (2006). \emph{Limiting laws of linear eigenvalue statistics
for unitary invariant matrix models}, {J. Math. Phys.} \textbf{47}, 103--303.

 \bibitem{Pa-Sh:11}  {Pastur, L.},  Shcherbina, M.  (2011). \textit{Eigenvalue Distribution of Large Random Matrices},  Mathematical Surveys and Monographs. Amer. Math. Soc., \textbf{171}.

\bibitem{Rudin}  Rudin, W. \emph{Real and complex analysis} (3rd). (1986). New York: McGraw-Hill Inc.

 \bibitem{Sh:11}  Shcherbina, M. (2011). \emph{Central limit theorem for linear eigenvalue
statistics of Wigner and sample covariance random matrices}, {J. Math.
Physics, Analysis, Geometry} \textbf{7}2,  176--192.

 \bibitem{T:16} Tieplova, D. (2017). \textit{ Distribution of eigenvalues of sample covariance matrices with tensor product samples}, 	 J. Math. Phys., Anal., Geom. {\bf 13}(1), 1-17. arXiv:1601.07443.


     \bibitem{Y}  Yaskov, P. (2014).  \textit{The universality principle for spectral distributions of sample covariance matrices}.  arXiv:1410.5190.

\end{thebibliography}
\end{document}